\documentclass[12pt,oneside]{amsart}
\usepackage{xspace}
\usepackage{graphicx}
\usepackage{latexsym}
\usepackage{mathtools}
\usepackage{appendix}
\usepackage{amsfonts,amsmath,amscd,amssymb}

\usepackage{hhline}
\usepackage{multirow}
\usepackage{framed}
\usepackage{mathrsfs}
\usepackage{tikz-cd}
\usepackage{enumitem}
\usepackage{hyperref}

\calclayout

\DeclareMathOperator{\Hom}{Hom}

\DeclareMathOperator{\tr}{tr}
\DeclareMathOperator{\Bil}{Bil}

\DeclareMathOperator{\End}{End}
\DeclareMathOperator{\GL}{GL}
\DeclareMathOperator{\Aut}{Aut}
\DeclareMathOperator{\Di}{Dic}
\DeclareMathOperator{\Stab}{Stab}
\DeclareMathOperator{\ResOp}{\downarrow}
\DeclareMathOperator{\IndOp}{\uparrow}

\newcommand{\catname}[1]{{\normalfont\textbf{#1}}}
\newcommand{\res}[2]{\ResOp_{#1}^{#2}}
\newcommand{\ind}[2]{\IndOp_{#1}^{#2}}

\newcommand{\Mod}[1]{\catname{${}_{#1}$Mod}}

\newcommand{\N}{\mathbb{N}}
\newcommand{\G}{\widehat{G}}

\newcommand{\F}{\mathcal{F}}
\newcommand{\R}{\mathbb{R}}

\newcommand{\C}{\mathbb{C}}

\newcommand{\CC}{\mathcal{C}}
\newcommand{\CR}{\mathfrak{C}}

\newcommand{\fn}[3]{#1 : #2 \rightarrow #3}
\newcommand{\twopartdef}[4]
{
	\left\{
		\begin{array}{ll}
			#1 & \mbox{if } #2 \\
			#3 & \mbox{if } #4
		\end{array}
	\right.
}
\newcommand{\CGG}{\C {\ast} \G}
\newcommand{\KGG}{{\mathbb K}\!\ast\!\G}

\newcommand{\bA}{{\mathbb A}}
\newcommand{\bB}{{\mathbb B}}
\newcommand{\bC}{{\mathbb C}}

\newcommand{\bF}{{\mathbb F}}

\newcommand{\bH}{{\mathbb H}}

\newcommand{\bK}{{\mathbb K}}

\newcommand{\bR}{{\mathbb R}}

\newcommand{\sA}{{\mathcal A}}
\newcommand{\sB}{{\mathcal B}}
\newcommand{\sC}{{\mathcal C}}
\newcommand{\sD}{{\mathcal D}}

\newcommand{\Id}{{\ensuremath{\mathrm{Id}}}}

\newcommand{\va}{\ensuremath{\mathbf{a}}\xspace}
\newcommand{\vb}{\ensuremath{\mathbf{b}}\xspace}

\newcommand{\ve}{\ensuremath{\mathbf{e}}\xspace}

\newcommand{\vg}{\ensuremath{\mathbf{g}}\xspace}
\newcommand{\vh}{\ensuremath{\mathbf{h}}\xspace}
\newcommand{\vi}{\ensuremath{\mathbf{i}}\xspace}
\newcommand{\vj}{\ensuremath{\mathbf{j}}\xspace}
\newcommand{\vk}{\ensuremath{\mathbf{k}}\xspace}

\newcommand{\vw}{\ensuremath{\mathbf{w}}\xspace}
\newcommand{\vx}{\ensuremath{\mathbf{x}}\xspace}
\newcommand{\vy}{\ensuremath{\mathbf{y}}\xspace}
\newcommand{\vz}{\ensuremath{\mathbf{z}}\xspace}

\newtheorem{thm}{Theorem}
\numberwithin{thm}{section}
\newtheorem{lem}[thm]{Lemma} 
\newtheorem{prop}[thm]{Proposition} 
\newtheorem{cor}[thm]{Corollary}

\theoremstyle{definition}
\newtheorem{defn}[thm]{Definition}
\newtheorem{eg}[thm]{Example}

\theoremstyle{remark}

\begin{document}
\title{Real Representations of $C_2$-Graded Groups: The Antilinear Theory}
\author{Dmitriy Rumynin}
\email{D.Rumynin@warwick.ac.uk}
\address{Mathematics Institute, University of Warwick, Coventry, CV4 7AL, UK
  \newline
\hspace*{0.31cm}  Associated member of Laboratory of Algebraic Geometry, National
Research University Higher School of Economics, Russia}
\author{James Taylor}
\email{james.h.a.taylor@outlook.com}
\address{Mathematics Institute, University of Warwick, Coventry,
  CV4 7AL, UK}
\date{June 15, 2020}
\thanks{The first author was partially supported within the framework of the HSE University Basic Research Program and the Russian Academic Excellence Project `5--100'.}
\thanks{We are indebted to to Matthew B. Young for useful conversations and interest in our work. We would like to thank Karl-Hermann Neeb for bringing the work of Wigner to our attention. We also thank the referee for the remark and the reference \cite{Faj} that we used in the introduction.}

\begin{abstract}
  We use the structure of finite-dimensional graded algebras 
  to develop the theory of antilinear representations of finite $C_2$-graded groups.
A finite $C_2$-graded group is a finite group with a subgroup of index 2.
In this theory the subgroup acts linearly, while the other coset acts antilinearly.
We introduce antilinear blocks, whose structure is a crucial component of the theory.
Among other things, we study characters and Frobenius-Schur indicators.
As an example, we describe the antilinear representations of the $C_2$-graded group $A_n \leq S_n$.
\end{abstract}

\maketitle


The first study of Real representation theory goes back to Wigner \cite{WIG}, where antilinear maps correspond to time reversal in Quantum Mechanics.
Independently it appears in equivariant KR-theory in works of Atiyah and Segal \cite{EKC} and Karoubi \cite{KAR}.
Following Atiyah \cite{BOT, KTR}, we use ``Real'' for objects with an involution and ``real'' for $\R$-related concepts. 

A $C_2$-graded group is a pair $G\leq\G$, where $G$ is a subgroup of index 2. 
Antilinear representations of $G$ were introduced by Wigner \cite{WIG} under the name of {\em corepresentations}: these are complex representations where $G$ acts linearly and the other coset $\G\setminus G$ acts antilinearly.
If $\G$ is the direct product $G\times C_2$,  the action of the element $(e_G,x)$, $x\neq e_{C_2}$ defines a complex conjugation commuting with the action of $G$. Hence, the Real representation theory becomes the real representation theory, a classical well-understood subject.

The most important work on the subject was written by Dyson \cite{Dyson}, although we found it only after finishing this paper.
Dyson used explicit matrix calculations to describe the structure of antilinear representations \cite{Dyson}. 
A number of studies over the years made progress on the split case (where $\G = G \rtimes C_2$, a semidirect product) \cite{EKC,KAR,Bor,ROZ,Fok}.
As pointed out by the referee, a mathematician only interested in the split case still needs the general case because $C_2$-graded subgroups of a split $C_2$-graded group are not split, in general (see \cite{Faj} where this fact plays an important role in KR-theory). 

The recent works by Neeb and \'{O}lafsson \cite{NEEB} and  Young \cite{R2R} reignited the interest in the general case.
Noohi and Young \cite{LTJ} gave a higher level treatment of the theory that included the case of twisted (projective) representations. 

Interestingly, the authors are often not aware of the past results, which appear in quite diverse areas.
Given this state of the art, our paper achieves several important aims. 
We develop a structural theory that gives clear conceptual proofs to all the results in the theory.
While doing it, we prove several new results.
Our comprehensive exposition will be useful even for those interested only in the split case.
Finally, we attempt to make accurate historic attributions for those of our results that previously appeared in the literature.

It is interesting that the Real 2-representation theory of $C_2$-graded groups \cite{R2R,BR2} is as developed as its 1-counterpart, the topic of the present paper. In 2-representation theory, the step from split to general $C_2$-graded groups goes back to Hahn \cite{HAH}, who extends hermitian Morita theory from algebras with involution to algebras with anti-structure.  In the split case the group algebra $\bC G$ becomes an algebra with involution and Real representation theory can be developed using modules over such algebras \cite{ROZ}. 
In the general case the group algebra $\bC G$ becomes an algebra with anti-structure. Since the representation theory of such algebras is not available, we work out Real representation theory from first principles.

Let us describe the content of the paper in detail. 
In Section~\ref{chap2} we introduce graded groups. 
For $C_2$ as the grading group, we describe the Real conjugation action and Real conjugacy classes.

In Section~\ref{chap3}, we develop the general theory of antilinear representations. Equivalently, these are modules over the skew group ring $\CGG$.
We call them A-representations and introduce A-characters (A stands for antilinear).
Next we prove Theorem~\ref{prop:centredim}, our first main result.
The A-character theory was also developed by Noohi and Young \cite{LTJ} but our proofs are substantially different.
To make the paper self-contained we include them. 

Next we describe Complexification and Realification, the functors that correspond to changing scalars in the classical theory. We also investigate the existence of invariant bilinear forms, proving our next main result, Theorem~\ref{biltypes}.
We finish the section with an explanation how Real theory specialises to real theory. 

In Section~\ref{Ch:blocks} we recall the classification of real $C_2$-graded division algebras and use it to describe real $C_2$-graded simple algebras.
We introduce the notion of an A-block.
Then we prove Dyson's Theorem (Theorem~\ref{possibleblocks}), describing the ten possible structures of an A-block.
These ten possible structures appear in Dyson's work \cite[Section IV]{Dyson} without the notion of an A-block.
His proof consists of matrix calculations, while we give a concise structural proof.

In Section~\ref{FS_indicator} we define the Real Frobenius-Schur indicator.
We learned this definition from Matthew B. Young.
The main result is the Bargmann-Frobenius-Schur Criterion (Theorem~\ref{FS}),
essentially asserting that the Real Frobenius-Schur indicator is the correct generalisation of the classical Frobenius-Schur indicator.
The result is already proved by Dyson \cite[Section VI]{Dyson}, who attributes it to a private communication from Bargmann.
In the split case this result was later proved by Borevi\v{c} and Devjatko \cite{Bor}, although we could not find the paper itself,
only the statement and its attribution in \cite{ROZ}. 
Notice that the standard classical proof does not generalise to Theorem~\ref{FS} because there is no analogue to the decomposition of bilinear forms into symmetric and alternating squares (cf. \ref{classical_split} for further details). 
Dyson's proof involves a direct calculation with tensors, while we deduce the result from the structure of A-block (Theorem~\ref{possibleblocks}) by an elementary calculation.

In Section~\ref{realchars} we undertake a further study of A-characters. 
The first main result, Theorem~\ref{dim_hom}, is the only overlap of this section with \cite{LTJ}, but our proof is again different. The second main result of this section is Theorem~\ref{table}, which brings together our earlier results and resembles similar summaries of the classical theory - see the appendix by the first author \cite{RUM}. The third main result is Theorem~\ref{square}. Its resemblance to its classical counterpart indicates that the Real theory is, indeed, a natural and useful generalisation of the real theory. 

Finally, in Section~\ref{A<S}, we work out the example of the symmetric Real structure on the alternating group $A_n$. 
We give an explicit description of the Real conjugacy classes in terms of cycle type. 
The Real symmetric theory of $A_n$ is different but intricately related to the real theory of $A_n$.

Throughout we use the following notation. $G$, $\G$ and $H$ are always finite groups, small bold letters $\vg$, $\vh$ etc. are group elements. $(\vh)_H$ is the conjugacy class of $\vh\in H$ and $C_H(\vh)$ is the centraliser. The group $C_2 = \{ 1, -1\}$ is multiplicative, $G \leq \G$ is an index 2 subgroup and $\vw$ is an element of $\G \setminus G$. 
$V$ is a $\C G$-module, $\vw \cdot V$ is the twisted $\C G $-module.
$W$ is an A-representation or a $\CGG$-module.

\section{$C_2$-Graded Groups}\label{chap2}
%
\subsection{Graded Groups}
Let $H$ be a group. By an \emph{$H$-graded group} we understand a short exact sequence
$$1 \rightarrow G \rightarrow \widehat{G} \xrightarrow{\pi} H \rightarrow 1,$$
where $G = \ker(\pi)$ is called the \emph{ungraded subgroup} of $\widehat{G}$.
\begin{eg}\label{eg:Hexamples}
\begin{enumerate}[label = (\roman*)]
\item
  For any group $G$, we have the \emph{standard graded group} $1 \rightarrow G \rightarrow G \times H \rightarrow H \rightarrow 1$. More generally, all graded groups are classified using Schreier's systems of factors.
\item
  Let $\bK$ be a field, $H$ a subgroup of its group of automorphisms. Let $V$ be a $\bK$-vector space.
  For each $\vh\in H$, consider the set $\GL^{\vh}(V)$ of invertible $\vh$-linear transformations. These are $\bK^H$-linear transformations $f$ such that $f(kv) = \vh(k)f(v)$ for all $k\in\bK$, $v\in V$. Their union is an $H$-graded group: 
  $$
  1 \rightarrow \GL(V) \rightarrow \GL^*(V)= \coprod_{\vh\in H}\GL^{\vh}(V) \xrightarrow{\pi} H \rightarrow 1\, .
  $$
\end{enumerate}
\end{eg}

A homomorphism of $H$-graded groups is a homomorphism of groups $\phi:\widehat{G}\rightarrow\widehat{K}$ such that $\pi (\phi(\vg))=\pi(\vg)$ for all $\vg\in \widehat{G}$. 

\subsection{Real Structures}
We refer to a $C_2$-graded group as a \emph{Real structure} on $G$. Here, $C_2=\{1,-1\}$ is the multiplicative cyclic group of order 2. We commonly refer to elements of $G$ as even and of $\G \setminus G$ as odd. 

\begin{eg}\label{eg:Z2examples}
\begin{enumerate}[label = (\roman*)]
\item
  We have four cyclic Real structures: $C_n \leq C_{n}\times C_2$, $C_n \leq C_{2n}$, $C_n \leq D_{2n}$ and $C_{2n} \leq \Di_{4n}$, where $\Di_{4n} = \langle \va ,\vx \mid \va^{2n}, \vx^2 = \va^n , \vx \va \vx^{-1} = \va^{-1} \rangle$ is the dicyclic group of order $4n$.
\item 
We call $A_n \leq S_n$ the symmetric Real structure on the alternating group $A_n$. 
\item
  Let $G$ be a group. Let $\text{Anti}(G)$ be the set of antiautomorphisms of $G$, and define $\Aut^*(G) = \Aut(G) \sqcup  \text{Anti}(G)$. Then $1 \rightarrow \Aut(G) \rightarrow \Aut^*(G) \rightarrow C_2 \rightarrow 1$ is a Real structure.  Note that $\Aut(G) = \text{Anti}(G)$ if and only if $\Aut(G) \cap \text{Anti}(G)\neq \emptyset$ if and only if $G$ is abelian. Hence, it is essential that the union is disjoint.
\end{enumerate}
\end{eg}
\subsection{Real Conjugacy Classes} \label{Real_CC}
If $G$ is the ungraded subgroup of a $C_2$-graded group $\G$, then there is an associated homomorphism of $C_2$-graded groups:
\begin{align*}
\fn{\psi}{\G}{\Aut^*(G)},  \quad \psi(\vz)(\vg) = \vz \vg^{\pi(\vz)}\vz^{-1}.
\end{align*}
The resulting action of $\G$ on $G$ is called \emph{Real conjugation}. We write $(\!(\vg)\!)$ for the Real conjugacy class of $\vg$, and $\CR_{\G}(\vg)$ for the stabiliser. Using the Orbit Stabiliser Theorem, we have the following, which applies to conjugation and Real conjugation both being actions of $\G$ on $G$ which extend the conjugation action of $G$.

\begin{lem}\label{lem:doubleaction}
Let $G \leq \G$ be an index $2$ subgroup, and $G$, $\G$ both act on a set $X$, with the action of $\G$ extending that of $G$. Then for any $x \in X$, we have two (mutually exclusive) cases:
\begin{enumerate}[label=(\roman*)]
\item
$(x)_{\G} = (x)_G \sqcup (\vz \cdot x)_G$ for any odd $\vz$, $\Stab_G(x) = \Stab_{\G}(x)$ and no odd element stabilises $x$.
\item
$(x)_{\G} = (x)_G$, $\Stab_G(x) \leq \Stab_{\G}(x)$ is an index $2$ subgroup, and $x$ is stabilised by some odd element.
\end{enumerate}
\end{lem}

Applied to the conjugation action of $\G$ on $G$, for all $\vg \in G$, the two cases are:
\begin{enumerate}[label=(\roman*)]
\item[A1]
$(\vg)_{\G} = (\vg)_G \sqcup (\vz \vg \vz^{-1})_G$ for any odd $\vz$, $C_G(\vg) = C_{\G}(\vg)$, and no odd element commutes with $\vg$.
\item[A2]
$(\vg)_{\G} = (\vg)_G$, $C_G(\vg) \overset{2}{\leq} C_{\G}(\vg)$ and $\vg$ commutes with some odd element.
\end{enumerate}

For example, this gives the well-known splitting criterion for the conjugacy classes of the alternating group. For the Real conjugation action of $\G$ on $G$, for all $\vg \in G$, the two cases are:
\begin{enumerate}[label=(\roman*)]
\item[B1]
$(\!(\vg)\!) = (\vg)_G \sqcup (\vz \vg^{-1} \vz^{-1})_G$ for any odd $\vz$, $C_G(\vg) = \CR_{\G}(\vg)$, and no odd element has $\vz \vg^{-1} = \vg \vz$.
\item[B2]
$(\!(\vg)\!) = (\vg)_G$, $C_G(\vg) \overset{2}{\leq} \CR_{\G}(\vg)$ and for some odd $\vz$, $\vz \vg^{-1} = \vg \vz$.
\end{enumerate}

So given $\vg \in G$, we have four mutually exclusive cases regarding what can happen, and each of these cases is known to occur. For a conjugacy class $X$ of $G$, the set $X^{-1}$ is another conjugacy class, which is either disjoint or equal. If equal, we call $X$ self-inverse. We have a nice description of the Real conjugacy classes when all conjugacy classes of $\G$ are self-inverse.

\begin{prop}\label{pr:selfinverseReal}
Let $G \leq \G$ be a Real structure with all conjugacy classes of $\G$ self-inverse. Then for $\vg \in G$, 
\[
(\!(\vg)\!) = \twopartdef{(\vg)_{G}}{(\vg)_{G} \text{ is not self-inverse,}}{(\vg)_{\G}}{(\vg)_{G} \text{ is self-inverse.}}
\]
\end{prop}
\begin{proof}
We have that $(\vg)_G \subset (\!(\vg)\!) \subset (\vg)_{\G} \cup (\vg^{-1})_{\G} = (\vg)_{\G}$. If $(\vg)_G$ is of type A2, then $(\vg)_G = (\vg)_{\G}$, and thus $(\!(\vg)\!) = (\vg)_{\G}$. Otherwise, $(\vg)_G$ is of type A1. If $(\vg)_G$ is not self-inverse, then as $(\vg)_{\G}$ is self-inverse, there is some odd $\vy$ element with $\vy\vg^{-1}\vy^{-1} = \vg$, so we are in case B2 and $(\!(\vg)\!) = (\vg)_{G}$. If $(\vg)_G$ is self-inverse, we are in case B1 (because no odd element commutes with $\vg$), thus $\left|(\!(\vg)\!)\right| = 2\left| (\vg)_G \right|$, hence $(\!(\vg)\!) = (\vg)_{\G}$.
\end{proof}


\section{The Antilinear Theory}\label{chap3}
\subsection{Graded and Antilinear Representations} \label{gr_anti_rep}
Let $\bK$ be a field, $H$ a subgroup of its group of automorphisms.
By {\em a graded representation} of an $H$-graded group $G\leq\G$,
we understand a pair $(V,\rho)$ where $V$ is a $\bK$-vector space, $\rho: \G\rightarrow \GL^\ast (V)$ is a homomorphism of $H$-graded groups. \emph{A homomorphism of graded representations} is a $\bK^H$-linear $\G$-equivariant map.

In other words, these are just modules over the skew group algebra $\KGG$, a $\bK$-vector space with a basis $\G$ and multiplication $a \vg \cdot b \vh = a \pi(\vg)(b) \vg\vh$. Notice that $\KGG$ is a $\bK^H$-algebra but not a $\bK$-algebra. Thus, the category of graded representations is $\bK^H$-linear. 
\begin{prop}\label{modulerepequiv}
The functors
\begin{align*}
& \: \KGG\text{-modules} \:& & &\longleftrightarrow& & &\: \text{Graded representations of } G \leq \G \:& \\
&(V, \: (a \vg) \cdot v \coloneqq a \rho(\vg)(v))& & &\longleftrightarrow& & &(V, \: av \coloneqq (a e)\cdot v, \: \rho(\vg)(v) \coloneqq \vg \cdot v)&,
\end{align*}
which are both the identity on morphisms, are isomorphisms of categories.
\end{prop}

If the characteristic of $\bK$ does not divide $|\G|$, then $\KGG$ is a semisimple (\cite[Cor. 0.2(1)]{MON}) $\bK^H$-algebra.
By the Artin-Wedderburn Theorem, $\KGG$ is isomorphic to a finite product of matrix algebras over finite-dimensional division $\bK^H$-algebras. It would be interesting to investigate this decomposition further. For instance, does Unger's algorithm \cite{Ung} work in this setting?

Our particular interest lies with $C_2$-graded groups and $\bR\leq \bC$ as $\bK^H\leq \bK$. 
In this case we refer to a $C_2$-graded representation as \emph{an antilinear representation} or  \emph{an A-representation}.
By \emph{an A-homomorphism} we understand of homomorphism of $C_2$-graded representations.
Notice that the set of all A-homomorphisms $\Hom_A(V,W)$ is a real vector space.

\subsection{Antilinear Maps and Matrices}
We use the notation
\begin{align*}
\prescript{\epsilon}{}{A} &= \twopartdef{A}{\epsilon = 1}{\overline{A}}{\epsilon = -1} & &\text{ and } &  \prescript{\epsilon}{}{\lambda} &= \twopartdef{\lambda}{\epsilon = 1}{\overline{\lambda}}{\epsilon = -1,}
\end{align*}
where $\epsilon \in C_2$, $A \in M_n(\C)$ and $\lambda \in \C$. Let us establish the correspondence between antilinear maps and matrices.
A choice of basis in an $n$-dimensional vector space $V$ gives a bijection
$$
\End_{\C}(V) \ni T \longleftrightarrow [T]\in M_n(\C)\, .
$$
A similar bijection exists for antilinear maps
$$
\End^{-1}_{\C}(V)\coloneqq \Hom_{\bC} (\overline{V},V) \ni U \longleftrightarrow [U]\in M_n(\C)\, ,
$$
where the chosen basis is used in both $V$ and $\overline{V}$. 
In elementary terms, a matrix $M$ determines an antilinear operator $v\mapsto M \overline{v}$ of $\bC^n$.

Define $\GL_n^*(\C) = \{A_1 \mid A \in \GL_n(\C)\} \sqcup \{A_{-1} \mid A \in \GL_n(\C)\}$ with multiplication 
\[
A_{\epsilon}B_{\delta} = (A \cdot {}^{\epsilon}B)_{\epsilon\delta} = \twopartdef{(A B)_{\epsilon\delta}}{\epsilon = 1}{(A \overline{B})_{\epsilon\delta}}{\epsilon = -1.}
\]
There is a canonical embedding $\GL_n(\C) \hookrightarrow \GL_n^*(\C)$ via $A \mapsto A_1$. Notice the inverses in this group:
\begin{equation} \label{inverses_group}
(A_{\epsilon})^{-1} = ({}^{\epsilon}(A^{-1}))_\epsilon \, ,
\mbox{ thus, }
(A_{1})^{-1} = (A^{-1})_1 \, , \ (A_{-1})^{-1} = (\overline{A^{-1}})_{-1} \, .
\end{equation}
The choice of basis of $V$ yields a graded group isomorphism $\GL^*(V)\xrightarrow{\cong}\GL_n^*(\C)$, sending
$Q \in \GL^*(V)$ to its matrix $[Q]_{\pi(Q)}\in\GL_n^*(\C)$. If $A$ is the change of basis matrix and $[Q]^\prime_{\pi(Q)}$ is the matrix of $Q$ in the new basis, then the change of basis formulas are 
$$
[T]_1 \Rightarrow [T]^{\prime}_1=(A [T] A^{-1})_1 \, , \ \
[U]_{-1} \Rightarrow [U]^{\prime}_{-1}=(A [U] \overline{A^{-1}})_{-1} \, .
$$
Finally, we define $A_{\epsilon}v \coloneqq A ({}^{\epsilon}v)$ for $A_{\epsilon} \in \GL_n^*(\bC)$ and $v \in \bC^n$.
%

\subsection{A-characters}
At this point it is convenient to choose a basis and to work with a matrix A-representation, a $C_2$-graded group homomorphism $\fn{\rho}{\G}{\GL^*(\C)}$. Its restriction to $G$ is a complex representation of $G$.
We define {\em the A-character} of $\rho$ as the character of $\rho|_{G}$:
$$
\fn{\chi}{G}{\C}\, , \ \chi(\vg) = \tr( \rho (\vg))\, .
$$
Notice that this formula gives only an ill-defined map $\G \rightarrow \C$ because the trace of $\rho(\vz)$ for $\vz \in \G \setminus G$ is not invariant under a basis change. On the bright side, the A-character uniquely determines the A-representation (Corollary~\ref{Achar_determines}). We call a function $\fn{\psi}{G}{\C}$ a \emph{Real class function} if $\psi$ is constant on Real conjugacy classes. 

\begin{lem}
A-characters are Real class functions.
\end{lem}
\begin{proof}
Let $\vg, \vh \in \G$. Then $[\vh \vg^{\pi(\vh)} \vh^{-1}]_1 = [\vh]_{\pi(\vh)}\cdot ({}^{\pi(\vh)}([\vg^{\pi(\vh)} \vh^{-1}]))_{\pi(\vh)} =  [\vh]_{\pi(\vh)} \cdot ({}^{\pi(\vh)}([\vg]^{\pi(\vh)}))_1\cdot ({}^{\pi(\vh)}[\vh^{-1}])_{\pi(\vh)}$. Therefore, as $\tr(AB) = \tr(BA)$ for matrices $A,B$, and $[\vh]_{\pi(\vh)}({}^{\pi(\vh)}[\vh^{-1}])_{\pi(\vh)} = (I)_1$, we have that
\begin{align*}
\tr([\vh \vg^{\pi(\vh)} \vh^{-1}]_1)
&= \tr( [\vh]_{\pi(\vh)} \cdot ({}^{\pi(\vh)}([\vg]^{\pi(\vh)}))_1\cdot ({}^{\pi(\vh)}[\vh^{-1}])_{\pi(\vh)}) \\
&= \tr( [\vh]_{\pi(\vh)} \cdot ({}^{\pi(\vh)}[\vh^{-1}])_{\pi(\vh)} \cdot ({}^{\pi(\vh)}([\vg]^{\pi(\vh)}))_1) \\
&= \tr(({}^{\pi(\vh)}([\vg]^{\pi(\vh)}))_1) = \tr([\vg]_1). \qedhere
\end{align*}
\end{proof}

\subsection{Operations on Antilinear Representations}
Standard constructions with complex representations extend to A-representations. Let $(W,\rho)$ and $U$ be A-representations.

The space of $\C$-linear maps
$\Hom_{\C}(U,W)$,
is an A-representation via $\vg \cdot f \coloneqq \vg f \vg^{-1}$. In particular, the dual $U^* = \Hom_{\C}(U,\C)$ of any A-representation is also an A-representation. In terms of matrices, if $\vg$ acts as a matrix $A_x$, then $\vg$ acts on the dual space as $((A^{-1})^T)_x$. 

The tensor product $U \otimes_{\C} W$ is an A-representation via $\vg(u \otimes w) \coloneqq \vg u \otimes \vg w$. If $U=W$, the decomposition into symmetric and alternating squares is compatible with the action of odd elements, so these are A-subrepresentations.

If $\vz \in \G$, then $\vz \cdot W$ is an A-representation
$$ \vz \cdot W\, , \ \vg\mapsto \rho (\vz \vg \vz^{-1})\, .$$
Notice that $\rho(\vz)$ is an isomorphism $W \rightarrow \vz \cdot W$ if $\vz \in G$.
However, if $\vz \not\in G$, then it is an isomorphism $\overline{W} \rightarrow \vz \cdot W$. Therefore, any A-character $\chi$ satisfies $\overline{\chi} = \vz \cdot \chi$. This is a clear necessary condition for a complex representation to be extendible to an A-representation, but it is not sufficient (cf. Theorem~\ref{biltypes}).

Suppose $\G$ acts on a set $X$, the action is denoted $\star$. Then the free vector space is an A-representation with $\vg \cdot x = \vg \star x$, extended linearly for even elements and antilinearly for odd elements.

This can be constructed in another way. Take the complex representation $(\C X , \mu)$ and notice that in the basis $X$ the matrix of $\mu (\vg)$ (written $[\mu(\vg)]$) is real valued (in fact, it is $\{0,1\}$-valued). 
Hence, the map $[\G \rightarrow \GL^*_n(\C)]:\vg \mapsto [\mu(\vg)]_{\pi(\vg)}$ is a homomorphism of graded groups
and an A-representation.

This trick works on any real-valued matrix  representation of $\G$. It is an example of the matrix version of the induction functor from $\bR \G$-modules to $\CGG$-modules, cf. \ref{compandreal}.

\subsection{Skew Group Algebra}
Let us examine the centre $Z(\CGG)$ of the skew group algebra $\CGG$, defined in Section~\ref{gr_anti_rep}
It has real dimension $r + 2s+ t$, where $r$, $s$ and $t$ are the number of irreducible A-representations of type $\R$, $\C$ and $\mathbb{H}$ respectively. To relate this quantity to the number of conjugacy classes of $G$, we need an explicit form of central elements. The following is a special case of \cite[Prop. 3.13]{LTJ}.

\begin{lem}\label{lem:centralcrit}
An element $x \in \CGG$ is central if and only if it is of the form $x = \sum_{\vg \in G} c_\vg \vg$, where $c_\vg = a_\vg + i b_\vg \in \C$, $a_\vg ,b_\vg\in\bR$ and for any $\vz \in \G$ and $\vg \in G$, $c_{\vz \vg \vz^{-1}} =  {}^{\pi(\vz)} c_\vg$, or equivalently, $a_{\vz \vg \vz^{-1}} = a_\vg$ and $b_{\vz \vg \vz^{-1}} = \pi(\vz) b_\vg$.
\end{lem}

\begin{proof}
The element $x = \sum_{\vg \in \G} c_\vg \vg$ is central if and only if for all $c \vz \in \CGG$ with $c \in \C$, $\vz \in \G$, $x (c \vz) = (c \vz)x$, which is equivalent to $c_{\vz \vg \vz^{-1}} ({}^{\pi(\vg)}c) = c ({}^{\pi(\vz)} c_\vg)$. Thus, if $\vg$ is odd, then $c_\vg = 0$, and if $\vg$ is even, then $c_{\vz \vg \vz^{-1}} =  {}^{\pi(\vz)} c_\vg$.
\end{proof}

\begin{lem}\label{lem:commuteequal0}
If $x = \sum_{\vg \in G} (a_\vg + i b_\vg) \vg$ is central and $\vg$ commutes with any odd element, then $b_\vg = 0$. 
\end{lem}
\begin{proof}
If $\vz$ is odd with $\vz \vg \vz^{-1} = \vg$, then $b_\vg = b_{\vz \vg \vz^{-1}} = \pi(\vz) b_\vg= - b_\vg$, so $b_\vg = 0$.
\end{proof}

Generally, if $C$ is a conjugacy class or Real conjugacy class, the sum of elements of $C$ is not central. For example in $C_n \leq D_{2n}$, all Real conjugacy classes have size one, so $a$ is an element of this form, but $\vb\va\vb^{-1} = \va^{-1} \neq \va$, so $\vb\va \neq \va\vb$. However, we can explicitly give a basis of the centre in terms of the conjugacy classes of $G$ in the proof of the following proposition, also proven in \cite[Cor. 13.6]{LTJ}. We will reprove this less explicitly, using A-characters, in section~\ref{realchars}.

\begin{thm}\label{prop:centredim}
$\dim_{\R}Z(\CGG) = \#($Conjugacy Classes of $G)$.
\end{thm}
\begin{proof}
Using \ref{lem:centralcrit} and \ref{lem:commuteequal0} above, a basis is given by
$$S(\vg) = \sum_{\vk \in (\vg)_{\G}} \vk$$
for all conjugacy classes $(\vg)_{\G}$, and
$$T(\vg) = i \left( \sum_{\vk \in (\vg)_{G}} \vk - \sum_{\vk \in (\vh)_{G}} \vk \right)$$
for $(\vg)_{\G}$ where $(\vg)_{\G}$ is of type A2 with $(\vg)_{\G} = (\vg)_{G} \sqcup (\vh)_{G}$. Therefore, the number of basis elements is one for each class which doesn't split, and two for each class that does, which is the number of conjugacy classes of $G$.
\end{proof}

Let us recall the notion of crossed product. Given an algebra $\bA$, the crossed product is the new algebra  
\begin{equation} \label{crossed_product}
\bA \sharp_{\phi,a} C_2 \coloneqq \bA \oplus \bA\varpi , \
\varpi^2=a, \ \varpi x = \phi (x) \varpi \mbox{ for all } x\in \bA
\end{equation}
where $a\in \bA$ and $\phi$ is an automorphism of $\bA$ such that the formula~\eqref{crossed_product} defines an associative algebra. The skew group algebra $\CGG$ is a crossed product in two different ways. First, picking $\vw\in\G\setminus G$ yields
\begin{equation} \label{crossed_product_CG}
\CGG \cong \bC G \sharp_{\xi, \vw^2} C_2 \, , \
\xi (x) = \vw x \vw^{-1} \mbox{ for all } x\in \bC G \, .
\end{equation}
Second, the imaginary unit $i\in\bC$ yields
\begin{equation} \label{crossed_product_RGhat}
\CGG \cong \bR \G \sharp_{\iota, -1} C_2 \, , \
\iota (x) = - i x i \mbox{ for all } x\in \bR \G \, .
\end{equation}
The automorphism $\iota$ is of order two, so that it determines $C_2$-grading on the algebra $\bR \G$:
\begin{equation} \label{crossed_product_grading}
  \bR \G_{+} 
  \coloneqq \{ x \, | \, \iota (x) = x \} = \bR {G}\, , \
  \bR \G_{-} 
  \coloneqq \{ x \, | \, \iota (x) = -x \} = \bR (\widehat{G}\setminus G) \, .
\end{equation}

\subsection{Complexification and Realification}\label{compandreal}
Given a Real structure $G \leq \G$, we have the following square of $\R$-algebra inclusions:
\begin{equation} \label{square_1}
\begin{tikzcd}
    \R G  \arrow[hookrightarrow]{r} \arrow[hookrightarrow]{d} & \C G  \arrow[hookrightarrow]{d} \\
    \R \G \arrow[hookrightarrow]{r}  & \CGG
  \end{tikzcd}
\end{equation} 
Let $R$ be a subring of $S$. We denote restriction and induction functors by $\res{R}{S}$ and $\ind{R}{S}=S\otimes_{R}-$ respectively. Since all extensions in \eqref{square_1} are Frobenius, we observe a strong form of Frobenius reciprocity.
\begin{prop}[Frobenius Reciprocity] \label{FR}
If $R\leq S$ is one of the extensions in \eqref{square_1}, then induction $\ind{R}{S}$ is both left and right adjoint to restriction $\res{R}{S}$.
\end{prop}

Let us examine the extensions in \eqref{square_1}, starting from $\bR G  \leq \bC G$.
For a real representation $U$ of $G$, its complexification is
$U_{\C} = \C \otimes_{\R} U \cong U\ind{\bR G}{\bC G}$.
For a complex representation $V$ of $G$, the realification is ${}_{\R}V = V\res{\bR G}{\bC G}$.

The second pair $\C G  \leq \CGG$ contains generalisations (cf. Section~\ref{standard_real}) for A-representations.
\begin{defn}
For a $\C G $-module $V$, the \emph{Realification} of $V$ is the $\CGG$-module $V\ind{\C G }{\CGG} = \CGG \otimes_{\C G } V$. For a $\CGG$-module $W$, the \emph{Complexification} of $W$ is the $\C G $-module $W\res{\C G }{\CGG}$.
\end{defn}

Let us contemplate the composition of Realification and Complexification. Recall that $\vw \cdot V$ is the $\C G $-module $V$ with a new action of $G$ by
$\vg \bullet v \coloneqq \vw \vg \vw^{-1}v$.

\begin{prop}\label{pr:cresind}
For any $\vw \in \G \setminus G$ and a $\C G $-module $V$, we have a natural isomorphism of $\bC G$-modules
$$
 V \ind{\C G }{\CGG} \res{\C G }{\CGG} \cong V \oplus \vw \cdot \overline{V} \, .
$$
\end{prop}
\begin{proof}
As a $\C G $-bimodule, $\CGG \cong \C G  \oplus \C[\G \setminus G]$. Therefore,
\begin{align*}
 V \ind{\C G }{\CGG} \res{\C G }{\CGG} &\cong (\C G  \otimes_{\C G } V) \oplus (\C[\G \setminus G] \otimes_{\C G } V) \\ &\cong V \oplus (\C[\G \setminus G] \otimes_{\C G } V).
\end{align*}
The inverse of the $\bC G$-module homomorphism
$$
\fn{f}{\C[\G \setminus G] \otimes_{\C G } V}{\vw \cdot \overline{V}}, \  f(\lambda \vy \otimes v) =\overline{\lambda}\vw \vy v
$$
is given by $f^{-1}(v) =\vw^{-1} \otimes v$. It follows that 
$$
\fn{h}{V \oplus \vw \cdot \overline{V}}{\CGG \otimes_{\C G } V}, \ h(v , w) = 1 \otimes v + \vw^{-1} \otimes w
$$
is the natural isomorphism we seek. 
\end{proof}


\begin{prop}\label{pr:cindres}
For any $\CGG$-module $W$, we have a natural isomorphism of $\CGG$-modules
\[
W \res{\C G }{\CGG} \ind{\C G }{\CGG}  \cong W \oplus W \, .
\]
\end{prop}
\begin{proof}
  Fix an odd $\vw$, so that $G = G \sqcup \vw^{-1} G$. Define a homomorphism
  $\fn{f}{\CGG \otimes_{\C G } W}{W \oplus W}$ by
  $$
  f ( (a\vg + b \vw^{-1} \vh) \otimes w) = \left( \frac{1}{2} (a\vg + b \vw^{-1}\vh)w , -\frac{i}{2} (a\vg - b \vw^{-1}\vh)w \right) \, .
  $$
Its inverse is $f^{-1}(v,w) =1 \otimes (v + i w) + \vw^{-1} \otimes \vw(v - i w)$.
\end{proof}

\begin{cor} \label{Achar_determines}
  If $W_1 \res{\C G }{\CGG} \cong W_2 \res{\C G }{\CGG}$, then $W_1\cong W_2$.
  Thus an A-representation is determined by its A-character.
\end{cor}

The above result was known to E.Wigner \cite[p. 344]{WIG}.


\subsection*{Matrix Form of Realification} Suppose that $V$ is a $\C G $-module, with a basis $\{e_1, ... ,e_n\}$. For $\vg \in G$, let $[\vg]$ be the corresponding matrix. By \ref{pr:cresind}, the vector space $V \ind{\C G}{\CGG}$ admits a complex basis $\{1 \otimes e_i , \vw^{-1} \otimes e_i \mid 1 \leq i \leq n\}$ for some fixed odd $\vw$. The matrices of $\vg \in G$ and $\vz \in \G\setminus G$ respectively are
\begin{equation}
\begin{pmatrix}
[\vg] & \mathbf{0} \\ \mathbf{0} & \overline{[\vw \vg\vw^{-1}]}
\end{pmatrix}_1, \text{ and }
\begin{pmatrix}
 \mathbf{0} & [\vz \vw^{-1}] \\ \overline{[\vw \vz]} & \mathbf{0} 
\end{pmatrix}_{-1}.
\end{equation}

Let us examine the third pair $\R \G \leq \CGG$, namely, the composition of the restriction and induction functors here.

\begin{prop}
For any $\CGG$-module $W$, we have a natural isomorphism of $\CGG$-modules
\[
W \res{\R \G}{\CGG} \ind{\R \G}{\CGG} \cong W \oplus \overline{W} \, .
\]
\end{prop}
\begin{proof}
  In one direction the isomorphism is
  $$
  \fn{f}{\CGG \otimes_{\R \G}W \res{\R \G}{\CGG}}{W \oplus \overline{W}}, \ 
  f((x+ i y) \otimes w) = ((x+ i y)w,(x- i y)w), \ x,y \in \R \G\, .
  $$
Its inverse is $f^{-1}(v,w) = \frac{1}{2} \left( 1 \otimes (w + v) + i \otimes (w - v) \right)$.
\end{proof}


Given a representation $X$ of $\G$, by $X \otimes \pi$ we denote the tensor product with the sign representation. Without loss of generality, $X \otimes \pi =X$ as vector spaces, with a new action $\vg \cdot v = \pi(\vg) \vg v$.
\begin{prop}
For any $\R \G$-module $X$, we have a natural isomorphism of $\bR \G$-modules
\[
X \ind{\R \G}{\CGG} \res{\R \G}{\CGG} \cong X \oplus (X \otimes \pi) \, . 
\]
\end{prop}
\begin{proof}
  In one direction the isomorphism is
  $$
  \fn{f}{\CGG \otimes_{\R \G}X}{X \oplus (X \otimes \pi)}, \
  f((x+i y) \otimes v) = (xv,yv), \  x,y \in \R \G \, .
  $$
  Its inverse is  $f^{-1}(v,w) = 1 \otimes v + i \otimes w$.
\end{proof}

\subsection*{Matrix Form of Induction} Suppose that $X$ is an $\R \G$-module, with basis $\{e_1, ... , e_n\}$. For $\vz \in \G$, let $[\vz]$ be the corresponding real valued matrix. A $\C$-basis of $W \coloneqq X \ind{\R \G}{\CGG}$ is $\{1 \otimes e_1, ... ,1 \otimes e_n\}$. In this basis, the matrix of $\vz$ on $W$ is $[\vz]_{\pi (\vz)}$.

\subsection*{Matrix Form of Restriction}
Let $W$ be a $\CGG$-module with $\C$-basis $\{e_1, ... , e_n\}$.
Let us choose the basis
$\{1 \otimes e_1, i \otimes e_1, ... , 1 \otimes e_n, i \otimes e_n\}$
of the restriction $X \coloneqq W \res{\R \G}{\CGG}$.
If an even $\vg$ has matrix $[\vg]_1$ on $W$, where $[ \vg ] = (a_{jk})$ and $a_{jk} = x_{jk} + i y_{jk}$, $x_{jk}, y_{jk}\in\bR$, then on $X$, $\vg$ has the matrix obtained from $[\vg]$ by replacing each $a_{jk}$ with the matrix
$\begin{psmallmatrix}
x_{jk} & -y_{jk} \\ y_{jk} & x_{jk}
\end{psmallmatrix}$.
For an odd element $\vz$ with matrix $[\vz]_{-1}$, where $[\vz] = (a_{jk})$ and $a_{jk} = x_{jk} + i y_{jk}$, 
the matrix of $\vz$ on $X$ is obtained by replacing each $a_{jk}$ with the matrix 
$\begin{psmallmatrix}
x_{jk} & y_{jk} \\ y_{jk} & -x_{jk}
\end{psmallmatrix}$.
Therefore, the matrix of $\vz$ on $X$ has zero trace. This observation is essential for relating the character to A-representations, see section~\ref{realchars}.

\subsection{Bilinear Forms}
We work with a $\C G $-module $V$ and a chosen element $\vw\in\G\setminus G$ in this section.
We call $V$ \emph{Realisable} if it is the restriction of some A-representation $W$. Let us characterise Realisable modules in terms of existence of certain bilinear forms.
We call a bilinear form $B$ on $V$ \emph{$\vw$-invariant} if 
\begin{equation} \label{invariant_form}
B(\vg u, \vw \vg \vw^{-1} v) = B(u,v) \mbox{ for all } \vg \in G, \; u,v\in V\, .
\end{equation}

\begin{thm}\label{biltypes}
A simple $\C G $-module $V$ falls into one of the three cases:
\begin{enumerate}[label=(\roman*)]
\item
$V \not\cong \vw \cdot V^{*}$ and there is no non-zero $\vw$-invariant bilinear form on $V$. 
\item
$V \cong \vw \cdot V^{*}$, $V$ is Realisable, and there exists a non-degenerate $\vw$-invariant bilinear form $B$ on $V$, which is \emph{$\vw$-symmetric}: $B(u,\vw^2 v) = B(v,u)$.
\item
$V \cong \vw \cdot V^{*}$, $V$ is not Realisable, and there exists a non-degenerate $\vw$-invariant bilinear form $B$ on $V$, which is \emph{$\vw$-alternating}: $B(u,\vw^2 v) = -B(v,u)$.
\end{enumerate}
Furthermore, the case that occurs is independent of the choice of $\vw \in \G \setminus G$.
\end{thm}
\begin{proof} 
We start with the following lemma, the proof of which is standard. 
\begin{lem}\label{3iso}
We have isomorphisms of $\bC G$-modules:   
\begin{align*}
&\Bil(V \times \vw \cdot V, \C) &  &\cong & &\Hom_{\C}(V,\vw \cdot V^*) & &\cong & &(V \otimes \vw \cdot V)^* ,\\
&B(u,v) & &\leftrightarrow & &F(u)(v) & &\leftrightarrow & &f(u \otimes v).
\end{align*}
\end{lem}

Note that $\Bil(V \times \vw \cdot V, \C)$ is the vector space $\Bil(V \times V, \C)$ of all bilinear forms on $V$ with a new (twisted in the second position) $\C G$-module structure. The lemma yields linear bijections between the corresponding subspaces of $G$-invariants:
\begin{equation} \label{3iso_inv}
\Bil(V \times \vw \cdot V, \C)^G  \cong \Hom_{\C G }(V,\vw \cdot V^*) \cong ((V \otimes \vw \cdot V)^*)^G \, . 
\end{equation}

Now suppose that $V \cong \vw \cdot V^{*}$.
By \eqref{3iso_inv} there exists a non-degenerate $\vw$-invariant bilinear form $B$.
On the other hand, $\hat{B}(u,v) \coloneqq B(v, \vw^{2}u)$ is another non-zero $\vw$-invariant bilinear form.
By Schur's Lemma, $B(u,v) = c B(v,\vw^{2}u)$ for some $c \in \C$. Observe that
$$
B(u,v) = c B(v,\vw^{2}u) = c( c B(\vw^2 u, \vw^2 v) ) = c^2 B(u,v)\, ,
$$
so that $c = \pm 1$ and $B(u,\vw^2v) = \pm B(v,u)$, thus, $B$ is either $\vw$-symmetric, or  $\vw$-alternating. 
By Schur's lemma and \eqref{3iso_inv}, $V$ cannot admit the $\vw$-symmetric and  $\vw$-alternating forms simultaneously. Thus, the theorem is reduced to Proposition~\ref{realisablebilinear}.
\end{proof}
\begin{prop}\label{realisablebilinear}
Let $(V,\rho)$ be a simple $\C G $-module, with $V \cong \vw \cdot V^{*}$. Then $V$ is Realisable if and only if there exists a non-degenerate $\vw$-invariant  bilinear form $B$ on $V$ with $B(u,\vw^2 v) = B(v,u)$.
\end{prop}
\begin{proof}
  Suppose that $V = W \res{\C G}{\CGG}$.
  Let $\langle \cdot, \cdot \rangle$ be a $G$-invariant hermitian inner product on $V$, which we assume to be linear in the first entry and semilinear in the second entry. Define $B(u,v) \coloneqq\langle u, \vw^{-1} v \rangle + \langle v, \vw u \rangle$.
  This is bilinear, as $\vw$ is antilinear, and $B(u,\vw u) = \langle u, u \rangle+ \langle \vw u, \vw u \rangle > 0$ for $u\neq 0$.
  Thus, $B$ is the required form.

  Conversely, suppose there exists such a bilinear form $B$.
  Let $\langle \cdot, \cdot \rangle$ be a $G$-invariant hermitian inner product on $V$.
  For each $u \in V$, the map $\fn{\langle \cdot , u \rangle}{V}{\C}$ is linear.
  Since $B$ is non-degenerate, there exists a unique $J(u) \in V$ with $B(\cdot , J(u)) = \langle \cdot , u \rangle$.
  The map $J$ is $\C$-antilinear and bijective since $\langle \cdot, \cdot \rangle$ is non-degenerate.
  In fact, it is an isomorphism of $\C G $-modules $\fn{J}{V}{\vw \cdot V}$, since $J(\vg u) = \vw \vg \vw^{-1} J(u)$. 
  Therefore, $\fn{\vw^{-2}J^2}{V}{V}$ is an isomorphism of $\C G$-modules.
  By Schur's Lemma, $\vw^{-2}J^2= r\Id_V$ for some $r \in \C$.
  Observe that
  $$
  \langle J(v) , J(v) \rangle = B(J(v),J^2(v)) 
= rB(J(v), \vw^{2} v) 
= rB(v,J(v)) 
= r \langle v,v \rangle,
$$
implying that $r \in \R_{>0}$.
Let $J^\prime \coloneqq r^{-1/2}J$. 
Clearly, $\vw^{-2}J^{\prime 2}=   r^{-1} \vw^{-2}J^2= \Id_V$.

Let us extend $\rho$ to odd elements by $\rho(\vg \vw) \coloneqq \rho(\vg) \circ J^{\prime}$, ensuring that the odd elements act antilinearly. It is easily verified that this extension is a homomorphism, using that $J^\prime \circ \rho(\vh) = \rho(\vw \vh \vw^{-1}) \circ J^\prime$ and, crucially, $\vw^{-2} J^{\prime 2} = 1$.
\end{proof}

We can relate Theorem \ref{biltypes} to subspaces of $(V \otimes \vw \cdot V)^*$. Define
\begin{equation}
  \fn{\tau}{(V \otimes \vw \cdot V)^*}{(V \otimes \vw \cdot V)^*}, \
  \tau(f)(u \otimes v) \coloneqq f(v \otimes \vw^2 u).
\end{equation}  
\begin{cor}\label{bildecomp}
Let $V$ be a simple $\C G $-module, with $V \cong \vw \cdot V^{*}$. Then
\[
((V \otimes \vw \cdot V)^*)^G =
\twopartdef{+1\mbox{-eigenspace of }\tau,}{V \text{ is Realisable,}}{-1\mbox{-eigenspace of }\tau,}{V \text{ is not Realisable.}}
\]
\end{cor}

\subsection{Standard Real Structure}\label{standard_real}
How does Real theory generalise the real representation theory of $G$?
The group $G$ admits the standard Real structure $G\leq G \times C_2$.
The following observation is immediate.
\begin{prop}\label{realantiequiv}
The following categories are equivalent: 
\begin{align*}
 \R\text{-representations of } G\:  & &\stackrel{\cong}{\longleftrightarrow}& & &\: \text{A-representations of } G \leq G \times C_2 \, .
\end{align*}
\end{prop}
\begin{proof}
On real representations the equivalence is given by 
\begin{align*}
(U,\mu) & &\longmapsto& & &(\C \otimes_{\R} U,\rho),\  \rho(\vg,\epsilon)(a \otimes u ) \coloneqq {}^{\epsilon}a \otimes \mu(\vg)(u)\, .
\end{align*}
On A-representations of $G \leq G \times C_2$ the equivalence is given by 
\begin{align*}
&(W,\mu) & &\longmapsto & &(W^{\times},\rho), \  W^{\times}\coloneqq \{w \in W \mid (\ve, -1)w = w\} \\
& & & & & \rho( \vg) (w) \coloneqq \mu(\vg,1) (w) = \mu (\vg,-1) (w) \, .  \qedhere
\end{align*}
\end{proof}

We denote the above functors by $\C \otimes_{\R} -$ and $(-)^{\times}$ respectively. The Complexification admits a simpler form for the standard Real structure.
\begin{lem}\label{stdcomplex}
  For a $\C G$-module $V$, there is a natural isomorphism of  $\bC {\ast} (G \times C_2)$-modules
  $$
  V\ind{\C G }{\bC {\ast} (G \times C_2) } \cong \C{\ast} C_2 \otimes_{\bC} V \, ,
  $$
    where $\bC {\ast} C_2 \otimes_{\C} V$ is a $\C {\ast} (G \times C_2)$-module via
    $a(\vg,\epsilon) \cdot (x \otimes v) = a(\ve,\epsilon)x \otimes \vg v$.
\end{lem}
\begin{proof} The isomorphism is written explicitly as
  $$
  \fn{f}{\C {\ast} (G \times C_2) \otimes_{\C G } V}{\C {\ast} C_2 \otimes_{\C} V}\, , \
  f(a(\vg, \epsilon) \otimes v) = a\epsilon \otimes \vg v
  $$
and its inverse is   $f^{-1}(r \otimes v) = r \otimes v$.
\end{proof}

Let us consider the following functor diagram:
\[
\begin{tikzcd}
    \Mod{\R G } \arrow[rr,yshift=.5ex,"\C \otimes_{\R} - "] \arrow[ddr,xshift=.75ex,"(-)_{\C}"] & & \Mod{\C {\ast} (G \times C_2)} \arrow[ll ,yshift=-.5ex,"(-)^{\times}"] \arrow[ddl, yshift=-.75ex,"\res{\C G }{\C {\ast} (G \times C_2)}"] \\ & & \\
    & \Mod{\C G } \arrow[uur, yshift=.6ex, xshift=.5, "{\C\ast C_2} \otimes_{\C} - "] \arrow[uul,xshift=-.75ex,"{}_{\R} (-)"] &
  \end{tikzcd}
\]
The next proposition asserts that Realification is a generalisation of realification.
\begin{prop}\label{realgen}
  For a $\C G $-module $V$, we have a natural isomorphism of $\bR G$-modules
  ${}_{\R} V \cong \left( \C {\ast} C_2 \otimes_{\C} V \right)^{\times}$.
\end{prop}
\begin{proof} The isomorphism is written explicitly as
  $$
  \fn{f}{{}_{\R} V}{\left( \bC {\ast} C_2 \otimes_{\C} V \right)^{\times}}\, , \  f(v) = ((\ve,1)+(\ve,-1)) \otimes v
$$
  and its inverse is   $f^{-1}((a(\ve,1) + b(\ve,-1)) \otimes v) = av$. 
\end{proof}

Let us now contemplate Complexification. 
\begin{prop}\label{compgen}
  For a $\C {\ast} (G \times C_2)$-module $V$, we have a natural isomorphism of $\bC G$-modules
  $V\res{\C G}{\C {\ast} (G \times C_2)} \cong (V^{\times})_{\C}$.
\end{prop}
\begin{proof}
  The isomorphism is written explicitly as
  $$
  \fn{f}{V\res{\C G}{\C {\ast} (G \times C_2)}}{\C \otimes_{\R}V^{\times}}\, , \
  f(v) = \frac{1}{2} ( 1 \otimes (v + (\ve,-1)v) - i \otimes (i v + (\ve,-1) i v))
  $$
  and its inverse is $f^{-1}(a \otimes v) = a v$.
\end{proof}

Finally, we can state a corollary asserting that Complexification is a generalisation of complexification.
\begin{cor}
  For an $\R G $-module $U$, we have a natural isomorphism of $\bC G$-modules
$U_{\C} \cong  (\C \otimes_{\R} U)\res{\C G }{\C\ast (G \times C_2)}$.
\end{cor}

\section{Antilinear Blocks} \label{Ch:blocks}
\subsection{Graded Division Algebras} \label{gr_division}
A $C_2$-graded algebra $\bF=\bF_{+} \oplus \bF_{-}$ is called graded division algebra if each non-zero homogeneous element is invertible.
Assuming that $\bF$ is real and finite-dimensional, we conclude that $\bF_{+}\in \{ \bR, \bC, \bH \}$. It is still possible that $\bF_{-}=0$ and the grading is trivial.

\begin{prop}
All real $C_2$-graded finite-dimensional division algebras with non-trivial grading are summarised in table~\ref{table_graded}. The underlying grading is unique up to isomorphism of $C_2$-graded algebras and can be given by $\iota$.
\end{prop}
\begin{table}
\caption{Classification of real $C_2$-graded division algebras}
\label{table_graded}
\begin{center}
\bgroup
\def\arraystretch{1.3}
\begin{tabular}{|c|c|c|r|c|c|c|} \hline
$\bF_{+}$  & $\bF$ & $x$ & $x^2$  & $\xi(a)$ & $\iota$  & $\bF\sharp_{\iota, -1} C_2$ \\ 
\hline 
$\bR$  & $\bR\oplus\bR$ & $(1,-1)$ & $1$  & $a$ & $(a,b)\mapsto(b,a)$ & $M_2(\bR)$\\
$\bR$  & $\bC$ & $i$ & $-1$  & $a$ & $a\mapsto \overline{a}$ & $\bH$ \\ 
$\bC$  & $\bC\oplus\bC$ & $(1,-1)$ & $1$  & $a$ & $(a,b)\mapsto(b,a)$ &  $M_2(\bC)$\\
$\bC$  & $M_2(\bR)$ & $\begin{pmatrix} 1 & 0  \\ 0 & -1 \end{pmatrix}$ & $1$  & $a$ & $\begin{pmatrix} a & b  \\ c & d \end{pmatrix}\mapsto\begin{pmatrix} d & -c  \\ -b & a \end{pmatrix}$ & $M_2(\bR\oplus\bR)$\\
$\bC$  & $\bH$ & $\vj$ & $-1$  & $\overline{a}$ & $a\mapsto -\vi a\vi$ & $\bH\oplus\bH$ \\
$\bH$  & $\bH\oplus\bH$ & $(1,-1)$ & $1$  & $a$ & $(a,b)\mapsto(b,a)$ &  $M_2(\bH)$\\
$\bH$  & $M_2 (\bC)$ & $\begin{pmatrix} 1 & 0  \\ 0 & -1 \end{pmatrix}$  & $1$  & $-\vk a \vk$ & $\begin{pmatrix} a & b  \\ c & d \end{pmatrix}\mapsto\begin{pmatrix} \overline{d} & -\overline{c}  \\ -\overline{b} & \overline{a} \end{pmatrix}$ & $M_4(\bR)$\\ 
\hline
\end{tabular}
\egroup
\vskip 3mm
\end{center}
\end{table}
\begin{proof}
The table represents $\bF$ as a crossed product $\bF_{+}\sharp_{\xi,x}C_2$ with convenient $x$ and $\xi$ chosen in the table. For consistency, we think of $\bC$ as either $\bR(\vi)$, or $\bR(i)$ and of $\bH$ as $\bR(\vi,\vj,\vk)$ where 
$$
i = \vi = \begin{pmatrix} 0 & 1  \\ -1 & 0 \end{pmatrix} \in M_2(\bR), \ 
\vj = \begin{pmatrix}0 & i\\i & 0\end{pmatrix},
\vk = \begin{pmatrix}i & 0\\0 & -i\end{pmatrix} \in M_2(\bC) \, .
$$

Since $\bF_{+}$ is a division algebra and $\bF\cong \bF_{+}\sharp_{\xi,x}C_2$ is a crossed product, $\bF$ is semisimple by Maschke Theorem for crossed products \cite[Th. 0.1.]{MON}. If $\bF$ is simple, then this is an application of the classification of real graded simple division algebras \cite{Rod,BaZa}.

    The remaining three non-simple examples in the table are the diagonal embeddings of $\bK\in \{ \bR, \bC, \bH \}$ into $\bK\oplus \bK$.
An arbitrary embedding $\pi = (\pi_1, \pi_2) : \bK \hookrightarrow \bK \oplus \bK$ is isomorphic to the diagonal embedding via the map $\pi_1\oplus\pi_2$.    

Nothing else is possible. If $\bF_{+}=\bR$, then $\bR\oplus\bR$ is the only 2-dimensional non-simple semisimple algebra.

If $\bF_{+}=\bC$, then $\bF$ cannot contain $\bR$ as a direct summand because the projection would give an impossible algebra map $\bC\rightarrow \bR$.
Thus, $\bC\oplus\bC$ is the only possible 4-dimensional non-simple semisimple algebra.

Finally, if $\bF_{+}=\bH$, then $\bF$ cannot contain $\bR$, $\bC$ or $M_2(\bR)$ as a direct summand because
again $\bH$ has no algebra maps to these algebras.
Thus, $\bH\oplus\bH$ is the only possible 8-dimensional non-simple semisimple algebra.
\end{proof}

The last column of table~\ref{table_graded} contains the algebra $\bF\sharp_{\iota, -1} C_2$ for our future perusal. 
Notice how the choice of $-1$ affects the crossed product.
For instance, $M_2 (\bC) \sharp_{\iota, 1} C_2 \cong M_2(\bH)$ and $M_2 (\bR) \sharp_{\iota, 1} C_2 \cong M_2(\bC)$.

\subsection{Graded Simple Algebra} \label{b_gsB}
Consider a finite-dimensional $C_2$-graded simple algebra $\bB$.
``Graded simple'' means that there are no non-zero proper graded ideals. 
By the graded Artin-Wedderburn theorem \cite[2.10.10]{NVO},
$\bB$ is graded algebra of matrices over either non-trivially $C_2$-graded or ungraded division algebra $\bF$.
A grading on the algebra of matrices is determined by an $n$-tuple $(\epsilon_1, \ldots \epsilon_n) \in C_2^n$:
$$
M^{gr}_{n} (\bF)_{+} = \{ (a_{i,j}) \,|\, a_{i,j} \in \bF_{\epsilon_i\epsilon_j^{-1}} \} \, , \, 
M^{gr}_{n} (\bF)_{-} = \{ (a_{i,j}) \,|\,  a_{i,j} \in \bF_{-\epsilon_i\epsilon_j^{-1}} \} \, .
$$
\begin{lem} \label{gr_matrices}
Suppose that the algebra $\bF$ is graded with $\bF_{-} = \bF_{+} x$ for an invertible $x$.
Then  $M^{gr}_{n} (\bF)_{+} \cong M_{n} (\bF_{+}).$
\end{lem}
\begin{proof}
  Without loss of generality, $\epsilon_1 = \ldots = \epsilon_m = -1$
  and
  $\epsilon_{m+1} = \ldots = \epsilon_n = 1$ for some $m$. Then 
  $M^{gr}_{n} (\bF)_{+}$ consists of the matrices of  the form
  $\begin{psmallmatrix} A & B x  \\ C x & D \end{psmallmatrix}$
  where all blocks have coefficients in $\bF_{+}$.
  Define the automorphism $\xi$ by $\xi(a) =  x a x^{-1}$.
  The explicit formula
  $$
  M_{n} (\bF)_{+} \rightarrow M_{n} (\bF_{+}), \
  \begin{pmatrix} A & Bx  \\ C x & D \end{pmatrix} \mapsto
  \begin{pmatrix} A & B  \\ \xi(C)x^2 & \xi(D) \end{pmatrix}
  $$
 yields the required isomorphism. 
\end{proof}

Suppose that $\bB$ is a real graded matrix algebra over a real non-trivially $C_2$-graded division algebra.
Clearly, we have an isomorphism of extensions
\begin{equation} \label{atom_gr}
\bB_{+} \leq \bB \leq \bB\sharp_{\iota, -1} C_2
\ \cong \ 
M_n \big(  \bF_{+} \leq \bF \leq \bF\sharp_{\iota, -1} C_2 \big)
\end{equation}
where $\bF_{+} \leq \bF$ occupies one of the  rows in table~\ref{table_graded}.

Now suppose that $\bB$ is a real graded matrix algebra over an ungraded division algebra with an invertible odd element.
In particular, $\bB=M_n(\bF)$,  $\bF\in \{ \bR, \bC, \bH \}$.
Any scalar matrix $a \Id$, $a\in \bF$ is fixed by the grading automorphism $\iota$.  By Noether-Skolem Theorem, the grading automorphism $\iota$ is a conjugation by $A\in M_{n} (\bF)$. Using Jordan forms (cf. \cite{Hua} for Jordan forms of quaternionic matrices), without loss of generality, $A$ is a diagonal matrix with a series of $-1$ followed by a series of $1$. It follows that
$$
M^{gr}_{n} (\bF)_{+} = \left(\!\!\begin{array}{c|c}\ast & 0\\\hline 0 & \ast\end{array}\!\!\right), \ \ 
M^{gr}_{n} (\bF)_{-} = \left(\!\!\begin{array}{c|c}0 & \ast\\\hline \ast & 0\end{array}\!\!\right). 
$$
Since, by assumption, $M^{gr}_{n} (\bF)_{-}$ has an invertible element,   
it follows that $n=2k$ and all blocks are of size $k\times k$. The following lemmas are immediate.
\begin{lem} \label{ungraded_lem1}
    $M^{gr}_{n} (\bF)_{+} \cong M_{k} (\bF \oplus \bF)$.
\end{lem}
\begin{lem} \label{ungraded_lem2}
  $M^{gr}_{n} (\bF) \sharp C_2  \cong M_{k} (M_2 (\bF \sharp C_2)) \cong M_{k} ( M_2(\frac{\bF[\varpi]}{(1+\varpi^2)}))$.
\end{lem}
The quotient algebra $\frac{\bF[\varpi]}{(1+\varpi^2)}$ is just the tensor product $\bF\otimes_{\bR}\bC$. It follows that, similarly to~\eqref{atom_gr}, we have an isomorphism of extensions
\begin{equation} \label{atom_unR}
\bB_{+} \leq \bB \leq \bB\sharp_{\iota, -1} C_2
\; \cong \; 
\begin{cases}
  M_k \big(  \bR\oplus\bR \leq M_2 (\bR) \leq M_2 (\bC) \big)  & \mbox{ if } \bF=\bR \, ,\\
  M_k \big(\bC \oplus \bC \leq M_2 (\bC) \leq M_2 (\bC \oplus \bC)\big) & \mbox{ if } \bF=\bC\, ,\\ 
  M_k \big(\bH \oplus \bH \leq M_2 (\bH) \leq M_4 (\bC)\big) & \mbox{ if } \bF=\bH\, . 
\end{cases}
\end{equation}

\subsection{Structure of an A-block}
The algebra $\bR \widehat{G}$ is graded (see~\eqref{crossed_product_grading}).
By the graded Artin-Wedderburn theorem \cite[2.10.10]{NVO}, the algebra is represented as
$$
\bR \widehat{G} = \oplus_k  M^{gr}_{n(k)} (\bF^k) \, ,
$$
a direct sum of $C_2$-graded matrices over graded division algebras $\bF^k$.
Recall that the graded matrices are simply matrices but there are different ways of equipping them with grading.

\begin{lem}\label{element}
  For each $k$  the component $M^{gr}_{n(k)} (\bF^k)_{-}$ contains an invertible element.
\end{lem}
\begin{proof}
If $\vw\in\widehat{G}\setminus G$, then  $\vw \Id_{M^{gr}_{n(k)} (\bF^k)}$ is such an element.
\end{proof}

Let us examine the four algebras in~\eqref{square_1}.
The conjugation by $\vw$ defines an automorphism $\xi$ of all four algebras (cf.~\eqref{crossed_product_CG}).
Let $e\in \bR{G}$ be a primitive central idempotent. Since $\vw^2\in G$, $\xi^2$ is an inner automorphism of $\bR G$ and $\xi^2(e)=e$. There are two cases to consider:
\begin{itemize}
\item{{\bf unsplit case:}} $\xi(e)=e$ so that $f\coloneqq e$ is central in $\CGG$,
\item{{\bf split case:}} $\xi(e)\neq e$ so that $f\coloneqq e + \xi (e)$ is central in $\CGG$.
\end{itemize}
Notice that in the split case $\vw e \neq e \vw$, so $e$ is not central already in $\bR \widehat{G}$.
By {\em an antilinear block} (or simply A-block) we understand the square, obtained from~\eqref{square_1} by the central idempotent $f$:
\begin{equation} \label{square_2}
\begin{tikzcd}
    \sA \coloneqq f\bR G  \arrow[hookrightarrow]{r} \arrow[hookrightarrow]{d} & \sC \coloneqq f\bC G  \arrow[hookrightarrow]{d} \\
    \sB \coloneqq f\bR \G \arrow[hookrightarrow]{r}  & \sD \coloneqq f\CGG
  \end{tikzcd}
\end{equation} 
The algebra $e \bR G $ is simple, hence, $e \bR G \cong M_n (\bF_a)$, where $\bF_a\in \{ \bR, \bC, \bH \}$. Depending on $\bF_a$, we will refer to real, complex or quaternion A-block. The algebra $\sC$ is always complex. The scalars in $\sB$ and $\sD$ are denoted $\bF_b$ and $\bF_d$. By $S_a$ ($S_b$, $S_c$ or $S_d$) we denote a simple module over $\sA$ ($\sB$, $\sC$ or $\sD$). A possible second simple module is denoted by $S^{\prime}_a\not\cong S_a$.  
Note that $\bF_a \cong \End_{\sA}S_a\cong \End_{\sA}S^{\prime}_a$. Ditto for $b$, $c$, $d$. 

Note that the algebras in an A-block are closely related, yet different from the algebras of Dyson $A\leq B \leq D$ \cite{Dyson}.
These are the images of $\bR G \leq \bC G \leq \CGG$ inside $\End_{\bR}S_d$.
Their structure can be deduced from the next theorem.

\begin{thm} (Dyson's Theorem) \label{possibleblocks}
  There are 10 possible structures of the A-block. They are
  summarised in table~\ref{table_block_structure}
  (including the numbers of distinct simple modules for each algebra, a small example of a graded group with its irreducible complex representation realising the given A-block, and a Dyson label DL (cf. \cite[IV]{Dyson})).
  The induction and restriction functors within each A-block are in table~\ref{table_block_induction}.
\end{thm}
\begin{table}
\caption{Possible A-block structures}
\label{table_block_structure}
\begin{center}
\bgroup
\def\arraystretch{1.3}
\begin{tabular}{|c|c|c|c|c|c|c|c|c|c|c|} \hline
& $\bF_a$  & $\bF_b$ & $\bF_d$ & $|\sA^{\vee}|$   & $|\sB^{\vee}|$  & $|\sC^{\vee}|$ & $|\sD^{\vee}|$ & $G\leq \widehat{G}$ & $S_c$ & DL \\ 
\hline 
I  & $\bR$ & $\bR$ & $\bR$  & $1$ & $2$ & $1$& $1$ & $C_1\leq C_2$ &  $\bC_{tr}$ & $RR$\\
II  & $\bR$ & $\bC$ & $\bH$  & $1$ & $1$ & $1$& $1$ & $C_2\leq C_4$ &  $\bC_{sn}$ & $QR$ \\
III  & $\bR$ & $\bR$ & $\bC$  & $2$ & $1$ & $2$&  $1$ & $K_4\leq D_8$ &  $\bC_{+}$ & $CR$ \\
IV & $\bC$ & $\bC$ & $\bC$  & $1$ & $2$ & $2$& $1$ & $C_3\leq C_6$ & $\bC_{\vw}$& $CC2$ \\
V & $\bC$ & $\bR$ & $\bR$  & $1$ & $1$ & $2$&  $2$ & $C_3\leq D_6$ &  $\bC_{\vw}$& $RC$\\
VI & $\bC$ & $\bH$ & $\bH$  & $1$ & $1$ & $2$&  $2$ & $C_4\leq Q_8$ & $\bC_{i}$ & $QC$\\
VII & $\bC$ & $\bC$ & $\bC$  & $2$ & $1$ & $4$&  $2$ & $C_8\leq C_8\rtimes C_2$ & $\bC_{\alpha}$& $CC1$ \\
VIII  & $\bH$ & $\bH$ & $\bH$  & $1$ & $2$ & $1$&  $1$ & $Q_8\leq Q_8\times C_2$ & $\bC^2$& $QQ$ \\
IX  & $\bH$ & $\bC$ & $\bR$  & $1$ & $1$ & $1$&  $1$ & $Q_8\leq Q_8\rtimes C_2$ & $\bC^2$ & $RQ$\\
X  & $\bH$ & $\bH$ & $\bC$  & $2$ & $1$ & $2$& $1$ & $Q_8\times C_2\leq G_{32}^8$ &  $\bC^2$& $CQ$ \\
\hline
\end{tabular}
\egroup
\vskip 3mm
\end{center}
\end{table}
\begin{table}
\caption{Induction and restriction within an A-block}
\label{table_block_induction} 
\begin{center}
\bgroup
\def\arraystretch{1.3}
\begin{tabular}{|c|c|c|c|c|c|c|c|c|} \hline
& $S_a\!\uparrow^{\sB}$  & $S_b\!\downarrow_{\sA}$ & $S_b\!\uparrow^{\sD}$ & $S_d\!\downarrow_{\sB}$   & $S_a\!\uparrow^{\sC}$ & $S_c\!\downarrow_{\sA}$   & $S_c\!\uparrow^{\sD}$  & $S_d\!\downarrow_{\sC}$ \\ 
\hline 
I  & $S_b\oplus S^{\prime}_b$ & $S_a$ & $S_d$  & $S_b\oplus S^{\prime}_b$ & $S_c$ & $2S_a$  & $2S_d$ &  $S_c$ \\
II  & $S_b$ & $2S_a$ & $S_d$  & $2S_b$ & $S_c$ & $2S_a$  & $S_d$ &  $2S_c$ \\
III  & $S_b$ & $S_a\oplus S^{\prime}_a$ & $S_d$  & $2S_b$ & $S_c$ & $2S_a$  & $S_d$ &  $S_c\oplus S^{\prime}_c$ \\
IV  & $S_b\oplus S^{\prime}_b$ & $S_a$ & $S_d$  & $S_b\oplus S^{\prime}_b$ & $S_c\oplus S^{\prime}_c$ & $S_a$  & $S_d$ &  $S_c\oplus S^{\prime}_c$ \\
V & $2S_b$ & $S_a$ & $S_d\oplus S^{\prime}_d$  & $S_b$ & $S_c\oplus S^{\prime}_c$ & $S_a$  & $2S_d$ &  $S_c$ \\
VI  & $S_b$ & $2S_a$ & $S_d\oplus S^{\prime}_d$  & $S_b$ & $S_c\oplus S^{\prime}_c$ & $S_a$  & $S_d$ &  $2S_c$ \\
VII  & $S_b$ & $S_a\oplus S^{\prime}_a$ & $S_d\oplus S^{\prime}_d$  & $S_b$ & $S_c\oplus S^{\prime}_c$ & $S_a$  & $S_d$ &  $S_c\oplus S^{\prime}_c$ \\
VIII  & $S_b\oplus S^{\prime}_b$ & $S_a$ & $S_d$  & $S_b\oplus S^{\prime}_b$ & $2S_c$ & $S_a$  & $S_d$ &  $2S_c$ \\
IX  & $2S_b$ & $S_a$ & $2S_d$  & $S_b$ & $2S_c$ & $S_a$  & $2S_d$ &  $S_c$ \\
X & $S_b$ & $S_a\oplus S^{\prime}_a$ & $2S_d$  & $S_b$ & $2S_c$ & $S_a$  & $S_d$ &  $S_c\oplus S^{\prime}_c$ \\
\hline
\end{tabular}
\egroup
\vskip 3mm
\end{center}
\end{table}
\begin{proof}
  Notice that $\iota (e)=e$,  $\iota (\vw)=-\vw$ and  $\iota (f)=f$.
It follows that $\sB$ is $C_2$-graded with $\sB_{+}=\sA$.   
Moreover, the element $f \vw \in\sB_{-}$ is invertible (cf.  Lemma~\ref{element}) so that
$$
\sB \cong \sA\sharp_{\xi , f\vw^2} C_2 \, .
$$

Let us prove that an A-block $\sB$ is graded simple. Suppose not. Then $\sA$ is not simple either, thus $\sA =M_n(\bF)\oplus M_n(\bF)$. Moreover, $\sB=h\sB\oplus (1-h)\sB$ for a central idempotent $h\neq 1$ and $h\sB$ is a graded simple algebra with $(h\sB)_{-} = h\sB_{-} \neq 0$.

The map $\theta :\sA\rightarrow \sB \rightarrow h\sB$ is a homomorphism of graded algebras. Thus, its image is in $(h\sB)_{+} = h\sB_{+}$, whose dimension is small: $\dim_{\bR}h\sB_{+} < \dim_{\bR} \sB_{+} = \dim_{\bR}\sA$. Hence, $\theta$ is not injective and its kernel is one of the direct summands of $\sA$, without loss of generality, $\xi(e)\sA$. It follows $\xi(e)h=0$ and $e\sA \leq h \sB$ and $e=h$. Similarly,  $\xi(e)\sA \leq (1-h) \sB$ and $\xi(e)=1-h$. This is a contradiction with the existence of non-zero element $h\vw \in\sB$:
$$
(h\vw)(1-h) = (h-h^2) \vw = 0
\mbox{ but }
(h\vw)(1-h) = e \vw \xi (e) = e \xi^2 (e) \vw = e\vw = h\vw \neq 0.
$$

Since $\sB$ is graded simple with an invertible odd element, we can apply the results of Section~\ref{b_gsB}.
Such analysis of the possible cases constitutes the rest of the proof.  
To show that a particular case is possible, it suffices to find a graded group $G\leq\widehat{G}$ with a homomorphism $\rho:\widehat{G}\rightarrow \sB$ such that $\rho (\widehat{G})$ $\bR$-spans $\sB$ and  $\rho(G)$ $\bR$-spans $\sA$.

The computation of $\sD$ is immediate from the two formulas
\begin{equation} \label{sD_compute}
\sD \cong  \sC \sharp_{\iota,-1} C_2
  \ \mbox{ and } \ 
M^{gr}_{n} (\bF) \sharp_{\iota,-1} C_2  \cong M_{n} (\bF\sharp_{\iota,-1} C_2)
\end{equation}
and the last column of table~\ref{table_graded}. 
The first isomorphism, given by $\varpi\mapsto i$, is similar to~\eqref{crossed_product_RGhat}. It works since $\sD$ is generated by $\sB$ and $i$.
The second isomorphism is given by the obvious formulas
  $(a_{i,j}) \mapsto(a_{i,j})$ and $(a_{i,j})\varpi \mapsto(a_{i,j}\varpi)$.

\subsection*{Unsplit Real A-block} 
It is clear that in this case 
\begin{equation}\label{urb_AC}
\bF_a=\bR, \;
S_a=\bR^n, \;
\sA=M_n (\bR), \;
S_c=\bC^n, \;
\sC=M_n (\bC) \, . 
\end{equation}
Consider a proper graded ideal $I=I_{+} \oplus I_{-}$ of $\sB$.
Since $\sB_{+}=\sA$ is simple, $I_{+}=0$. Since $f\vw$ is invertible and
$f\vw I_{-}\subseteq I_{+}$, $I_{-}=0$.
Hence, $\sB$ is graded simple
and the results of section~\ref{b_gsB}
are applicable.
In particular, we just need to apply formula~\eqref{atom_gr}
to $\bF_{+} \leq \bF$ occupying one of the top two rows 
in table~\ref{table_graded}. This yields the following two possibilities.
Notice that the roman numeral of each possibility corresponds to the row numbers in tables~\ref{table_block_structure} and \ref{table_block_induction}. 
\begin{itemize}
\item[(I)]
$\bF_b=\bR$,  so that
$S_b=\bR^n$,
$\sB=M_n (\bR)\oplus M_n (\bR)$, 
$S_d=\bR^{2n}$,
  $\sD=M_{2n}(\bR)$.
  The trivial representation of $C_1 \leq C_2$ yields such A-block with $n=1$.
\item[(II)]  $\bF_b=\bC$,  so that
$S_b=\bC^n$,
$\sB=M_n (\bC)$, 
$S_d=\bH^{n}$,
$\sD=M_{n}(\bH)$.
The sign representation of $C_2 \leq C_4$ yields such A-block with $n=1$. 
\end{itemize}

\subsection*{Split Real A-block} 
It is clear that in this case 
\begin{equation}\label{srb_AC}
\bF_a=\bR, \;
S_a=\bR^n, \;
\sA=M_n (\bR\oplus \bR), \;
S_c=\bC^n, \;
\sC=M_n (\bC\oplus \bC) \, . 
\end{equation}
Since $\sB$ is graded simple, the only possibility comes from \eqref{atom_unR}, where we replace $k$ with $n$ for consistency:
\begin{itemize}
\item[(III)] $\bF_b=\bR$,  so that
$S_b=\bR^{2n}$,
$\sB=M_{2n} (\bR)$, 
$S_d=\bC^{2n}$,
  $\sD=M_{2n}(\bC)$.
  If $\bC^2$ is the irreducible representation of $D_4$, then $\bC^2\!\downarrow_{K_4}=\bC_{+}\oplus\bC_{-}$
and $\bC_+$ is a representation of $K_4 \leq D_4$, which yields such A-block with $n=1$.
\end{itemize}


\subsection*{Unsplit Complex A-block} 
It is clear that in this case 
\begin{equation}\label{ucb_AC}
\bF_a=\bC, \;
S_a=\bC^n, \;
\sA=M_n (\bC), \;
S_c=\bC^n, \;
\sC=M_n (\bC\oplus\bC) \, . 
\end{equation}
It remains to apply formula~\eqref{atom_gr}
to $\bF_{+} \leq \bF$ occupying one of the three middle rows 
in table~\ref{table_graded}. This yields the following three possibilities.
\begin{itemize}
\item[(IV)]
$\bF_b=\bC$,  so that
$S_b=\bC^n$,
$\sB=M_n (\bC)\oplus M_n (\bC)$, 
$S_d=\bC^{2n}$,
  $\sD=M_{2n}(\bC)$.
The non-trivial representation $\bC_{\vw}$, $\vw=e^{2\pi i/3}$ of $C_3 \leq C_6$ yields such A-block with $n=1$.  
\item[(V)]  $\bF_b=\bR$,  so that
$S_b=\bR^{2n}$,
$\sB=M_{2n} (\bR)$, 
$S_d=\bC^{2n}$,
  $\sD=M_{2n}(\bC)$.
The non-trivial representation $\bC_{\vw}$ of $C_3 \leq D_6$ yields such A-block with $n=1$.    
\item[(VI)]  $\bF_b=\bH$,  so that
$S_b=\bH^n$,
$\sB=M_n (\bH)$, 
$S_d=\bH^{n}$,
$\sD=M_{n}(\bH\oplus\bH)$.
The faithful representation $\bC_{i}$ of $C_4 \leq C_8$ yields such A-block with $n=1$.  
\end{itemize}

\subsection*{Split Complex A-block} 
It is clear that in this case 
\begin{equation}\label{scb_AC}
\bF_a=\bC, \;
S_a=\bC^n, \;
\sA=M_n (\bC\oplus \bC), \;
S_c=\bC^n, \;
\sC=M_n (\bC\oplus \bC\oplus \bC\oplus \bC) \, . 
\end{equation}
Since $\sB$ is graded simple, the only possibility comes from \eqref{atom_unR}, where we replace $k$ with $n$ for consistency:
\begin{itemize}
\item[(VII)] $\bF_b=\bC$,  so that
$S_b=\bC^{2n}$,
$\sB=M_{2n} (\bC)$, 
$S_d=\bC^{2n}$,
  $\sD=M_{2n}(\bC\oplus\bC)$.
  The faithful representation $\bC_{\alpha}$, $\alpha=e^{\pi i/4}$ of $C_8 \leq C_8\rtimes C_2$ yields such A-block with $n=1$.
  The semidirect product is given by the automorphism $\varphi (\vx)=\vx^5$ of $C_8$.
  It is also known as {\em the modular group of order 16} or $G_{16}^{6}$ in the list of small groups.
\end{itemize}

\subsection*{Unsplit Quaternionic A-block} 
It is clear that in this case 
\begin{equation}\label{uqb_AC}
\bF_a=\bH, \;
S_a=\bH^n, \;
\sA=M_n (\bH), \;
S_c=\bC^{2n}, \;
\sC=M_{2n} (\bC) \, . 
\end{equation}
It remains to apply formula~\eqref{atom_gr}
to $\bF_{+} \leq \bF$ occupying one of the two bottom rows 
in table~\ref{table_graded}. This yields the following two possibilites.
\begin{itemize}
\item[(VIII)]
$\bF_b=\bH$,  so that
$S_b=\bH^n$,
$\sB=M_n (\bH)\oplus M_n (\bH)$, 
$S_d=\bH^{2n}$,
  $\sD=M_{2n}(\bH)$.
  The faithful representation $\bC^2$ of $Q_8 \leq Q_8\times C_2$ yields such A-block with $n=1$.
  The direct product is also known as {\em the Hamiltonian group of order 16} or $G_{16}^{12}$ in the list of small groups.
\item[(IX)]
  $\bF_b=\bC$,  so that
$S_b=\bC^{2n}$,
$\sB=M_{2n} (\bC)$, 
$S_d=\bR^{4n}$,
$\sD=M_{4n}(\bR)$.
  The faithful representation $\bC^2$ of $Q_8 \leq Q_8\rtimes C_2$ yields such A-block with $n=1$.
  The semidirect product is also known as {\em the Pauli group} or $G_{16}^{13}$ in the list of small groups.
\end{itemize}

\subsection*{Split Quaternionic A-block} 
It is clear that in this case 
\begin{equation}\label{sqb_AC}
\bF_a=\bH, \;
S_a=\bH^n, \;
\sA=M_n (\bH\oplus \bH), \;
S_c=\bC^{2n}, \;
\sC=M_{2n} (\bC\oplus \bC) \, . 
\end{equation}
Since $\sB$ is graded simple, the only possibility comes from \eqref{atom_unR}, where we replace $k$ with $n$ for consistency:
\begin{itemize}
\item{(X)} $\bF_b=\bH$,  so that
$S_b=\bH^{2n}$,
$\sB=M_{2n} (\bH)$, 
$S_d=\bH^{2n}$,
  $\sD=M_{2n}(\bH\oplus\bH)$.
  The representation $\bC^2$ of $Q_8 \leq Q_8\rtimes G_{32}^8$, coming from the faithful representation of $Q_8$, yields such A-block with $n=1$.
\end{itemize}
\end{proof}

\section{Frobenius-Schur Indicator} \label{FS_indicator}
By $\F_{\C}$ and $\F_{\R}$ we denote the Frobenius-Schur indicators of complex and real representations of $G$ respectively:
\[
\F_{\C} (V) \coloneqq \frac{1}{\left| G \right|}\sum_{\vx \in G} \tr_{\C}(\rho_V (\vx^2)), \
\F_{\R}(W) \coloneqq \frac{1}{\left| G \right|}\sum_{\vx \in G} \tr_{\R}(\rho_W (\vx^2)) \ 
\]
where $V$ is a $\C G$-module and $W$ a $\R G$-module. Similarly $\widehat{\F}_{\C}$ and $\widehat{\F}_{\R}$ will be the Frobenius-Schur indicators of $\G$.

Let $\chi$ be a character of a complex representation $V$ of $G$.
We define the $\emph{Real Frobenius-Schur indicator}$ of $\chi$ or $V$ as
\[
\F (V) = \F(\chi) \coloneqq \frac{1}{\left| G \right|}\sum_{\vz \in \G \setminus G} \chi(\vz^2).
\]
Notice that for the standard Real structure $\F=\F_{\bC}$.
\begin{prop}\label{fsrelation}
For a $\C G $-module $V$, $\F(V) = \frac{1}{2} \widehat{\F}_{\R}(V \res{\R G }{\C G } \ind{\R G }{\R \G}) - \F_{\C}(V)$.
\end{prop}
\begin{proof}
Notice that $\sum_{\vg} \chi(\vg^2) = \sum_{\vg} \chi(\vw \vg^2 \vw^{-1}) = \sum_{\vg} \overline{\chi(\vg^2)}$, and, hence, 
\[
\widehat{\F}_{\R}(V \res{\R G }{\C G } \ind{\R G }{\R \G}) = \frac{4}{\left| \G \right|} \sum_{\vg \in \G} \chi(\vg^2),
\]
because the character of $V \res{\R G }{\C G } \ind{\R G }{\R \G}$ is  $\chi + \overline{\chi} + \vw \cdot \chi + \vw \cdot \overline{\chi}$. Therefore
\begin{align*}
\F(V) &= \frac{2}{\left| \G \right|} \sum_{\vg \in \G} \chi(\vg^2) - \frac{1}{\left| G \right|}\sum_{\vg \in G} \chi(\vg^2) = \frac{1}{2} \widehat{\F}_{\R}(V \res{\R G }{\bC G} \ind{\R G}{\R \G}) - \F_{\C}(V). \qedhere
\end{align*}
\end{proof}

\begin{thm} (Bargmann-Frobenius-Schur Criterion) \label{FS}
  Let $\chi$ be the character of a simple $\bC G$-module $V$.
  If $W$ is a simple $\CGG$-module in the same A-block as $V$, then
\begin{align*}
\F(\chi) =
\left\{
		\begin{array}{cl}
			1 & \mbox{if } W \mbox{ is of type } \bR, \\
			0 & \mbox{if } W \mbox{ is of type } \bC,  \\
			-1 & \mbox{if } W \mbox{ is of type } \bH.
		\end{array}
	\right.
\end{align*}
\end{thm}

\begin{proof}
Using the above, $\F$ returns these values for the 10 cases described in \ref{possibleblocks}, as shown in the following table. The type of $W$ is the division algebra $\bF_d$.

\begin{center}
\begin{tabular}{|c|c|c|c|c|c|c|c|c|c|c|}
\hline
& I  & II &  III &  IV &  V &  VI &  VII &  VIII &  IX &  X  \\ 
\hline 
$\bF_d$  & $\bR$ & $\bH$ & $\bC$ & $\bC$ & $\bR$ & $\bH$ & $\bC$ & $\bH$ & $\bR$ & $\bC$ \\
$\widehat{\F}_{\R}(V \res{}{} \ind{}{})$  & $4$ & $0$ & $2$ & $0$ & $2$ & $-2$ & $0$ & $-4$ & $0$ & $-2$ \\
$\F_{\C}(V)$ & $1$ & $1$ & $1$ & $0$ & $0$ & $0$ & $0$ & $-1$ & $-1$ & $-1$ \\
$\F(V)$ & $1$ & $-1$ & $0$ & $0$ & $1$ & $-1$ & $0$ & $-1$ & $1$ & $0$ \\
\hline
\end{tabular}
\end{center}
\end{proof}

Although this theorem generalises the classical result to the Real setting, the original classical result is part of the proof. The proof also suggests that {\em the A-type} of $V$ should be the structure of the corresponding A-block. Thus, the A-type should be a roman numeral, I to X.
\begin{cor}
  The pair of Frobenius-Schur indicators $(\F_{\bC}(V),\F (V))$ nearly distinguish the A-type of $V$: it distinguishes eight types and a group of two remaining types (IV and VII). 
\end{cor}

\section{Properties of A-Characters}\label{realchars}
\subsection{Homomorphisms and Inner Products} 
Let $W$, $W_1$, $W_2$ be A-representations.
The invariants $W^{\G} = \{w \in W \mid \vg w = w$ for all $\vg \in \G\}$ do not generally form an A-subrepresentation, only an $\R \G$-submodule of dimension
$$
\dim_{\R}W^{\G} = \frac{1}{\left| \G \right|}\sum_{\vg \in \G} \psi(\vg) = \langle \psi , 1 \rangle_{\G} \, ,
$$
where $\psi$ is the real character of $W$. By definition, $\Hom_{\C}(W_1,W_2)^{\G} = \Hom_A(W_1,W_2)$.

\begin{thm}\label{dim_hom}
If $W_1$, $W_2$ are A-representations with A-characters $\chi_1$, $\chi_2$, then
\[
\dim_{\R}\Hom_A(W_1,W_2) = \langle \chi_1 , \chi_2 \rangle,
\]
where 
$\langle \cdot , \cdot \rangle$ is the inner product of class functions on $G$.
\end{thm}
\begin{proof}
  Let $\chi_U$ be the A-character of $U \coloneqq \Hom_{\C}(W_1,W_2)$, $\psi_U$ the character of $U$ as an $\R \G$-module. Then
\begin{align*}
\psi_U(\vg) = \twopartdef{\chi_U(\vg) + \overline{\chi_U(\vg)}}{\vg \in G,}{0}{\vg \not\in G,}
\end{align*}
as in the matrix form of restriction, the trace is 0 outside of $G$. Thus,
\begin{align*}
\dim_{\R}\Hom_A(W_1,W_2) &= \dim_{\R} \Hom_{\C}(W_1,W_2)^{\G}
\\ &= \frac{1}{\left| \G \right|}\sum_{\vg \in \G} \psi_U(\vg)
= \frac{1}{2\left| G \right|}\sum_{\vg \in G} \chi_U(\vg) + \overline{\chi_U(\vg)} \\ &= \frac{1}{2\left| G \right|}\sum_{\vg \in G} \overline{\chi_1(\vg)}\chi_2(\vg) + \chi_1(\vg)\overline{\chi_2(\vg)} \\
&= \frac{1}{2}\left(\langle \chi_1 , \chi_2 \rangle + \langle \chi_2 , \chi_1 \rangle \right) = \langle \chi_1 , \chi_2 \rangle. 
\end{align*}
Notice that the last equality holds because $\chi_1$ and $\chi_2$ are characters and, hence, $\langle \chi_1 , \chi_2 \rangle$ is a non-negative integer. 
\end{proof}

Let $S_j$, $j=1,\ldots, n\;$ be a complete set of non-isomorphic irreducible $A$-repre{\-}sen{\-}tations, $m_j = \dim_{\R}\End_A(S_j)$, $\chi_j = \chi_{S_j}$.
By Corollary~\ref{Achar_determines}, the A-characters $S_j$ are distinct but we can say more now.
Consider $\CGG \cong \oplus_j S_j^{b_j}$ as a left module over itself, with {\em the regular A-character} $\chi_{R}$. 
\begin{cor}\label{innerprodprop}
Let $W = \oplus_j S_j^{a_j}$ be a decomposition of an A-representation $W$ into irreducible A-subrepresentations. The following statements hold.
\begin{enumerate}[label=(\roman*)]
\item
$\chi_j$, $j=1,\ldots, n$ are orthogonal with respect to $\langle\cdot,\cdot\rangle$,   
\item
$\langle \chi_{j} , \chi_W \rangle = a_j m_j$,
\item
$\langle \chi_W , \chi_W \rangle = \sum_j a_j^2m_j$,
\item
$\langle \chi_{j} , \chi_{R} \rangle = 2 \dim_{\C}S_j$,
\item
$b_j = \frac{2 \dim_{\C}S_j}{m_j}$,
\item
$\left| G \right| = \sum_j \frac{1}{m_j} (\dim_{\C}S_j)^2$.
\end{enumerate}
\end{cor}

\subsection{Complexification and Realification Revisited}
The following proposition follows from computation of the $\R$-dimensions
in the Frobenius reciprocity (Proposition~\ref{FR}). By $\chi_{\cdot}$ we denote the corresponding complex or A-character. By $\psi_{\cdot}$ we denote the corresponding real character. We also skip subscripts and superscripts at arrows if they are unambiguous.
\begin{prop}\label{FR1/2}
  Let $V$ be a $\C G $-module, $W$ a $\CGG$-module,  $X$ a $\R \G$-module.
  The following statements hold. 
\begin{enumerate}[label=(\roman*)]
\item
  $\langle \chi_{V\ind{}{}} , \chi_W \rangle = 2 \langle \chi_{V} , \chi_{W\res{\bC G}{}} \rangle$
  and
  $\langle \chi_{V\ind{}{}} , \chi_{V\ind{}{}}  \rangle = 2(\langle \chi_V , \chi_V \rangle + \langle \chi_V , \chi_{\vw \cdot \overline{V}} \rangle)$.
\item
  $\langle \chi_{X\ind{}{}} , \chi_W \rangle = \langle \psi_{X} , \psi_{W\res{\bR G}{}} \rangle$ and $\langle \chi_{X\ind{}{}} , \chi_{X\ind{}{}}  \rangle = \langle \psi_{X} , \psi_{X} \rangle + \langle \psi_{X} , \psi_{X \otimes \pi} \rangle$.
\end{enumerate}
\end{prop}

The next theorem is an important clarification of the structure of an A-block, fusing together Theorems \ref{biltypes} and \ref{possibleblocks}. The distinction between the three types and the Realisability in each case was known to Neeb and \'{O}lafsson \cite[Thm. 2.11]{NEEB}.
This clarifies the aspects of an A-block, controlled by the Real Frobenius-Schur indicator. It also generalises a well-known result in the classical real theory (cf. \cite{RUM}).
\begin{thm}\label{table}
Let $W$ be an irreducible A-representation, $V$ an irreducible subrepresentation of $W\res{}{}= W\res{\C G }{\CGG}$. Let $V \ind{}{} = V\ind{\C G }{\CGG}$, $\vw$ a fixed odd element. Then $W$ and $V$ are as in one of the columns of table~\ref{table_classify}.
\begin{table}
\caption{The three types of A-representation}
\label{table_classify}
\begin{center}
\begin{tabular}{ c|c|c|c } 
 $\End_A(W)$ & $\R$ & $\C$ & $\mathbb{H}$ \\ 
 \hhline{=|=|=|=}
  $\End_{\C G }(W)$ & $\C$ & $\C \times \C$ & $M_2(\C)$ \\ 
 \hline
 $W\res{}{}$ & $V$ & $V \oplus \vw \cdot \overline{V}$ & $V \oplus V$ \\ 
 \hline
 $V\ind{}{}$ & $W \oplus W$ & $W$ & $W$ \\ 
 \hline
 $\dim_{\C}W$ & $n$ & $2n$ & $2n$ \\ 
 \hline
 $\dim_{\C}V$ & $n$ & $n$ & $n$ \\ 
 \hline
 $V \cong \vw \cdot \overline{V}$? & Yes & No & Yes \\ 
 \hline
 $V$ Realisable? & Yes & No & No \\ 
 \hline
 $\exists \: \vw$-invariant bilinear form? & Yes ($\vw$-sym.) & No & Yes ($\vw$-alt.) \\ 
 \hline
 $\F(V)$ & 1 & 0 & $-1$ \\
\end{tabular}
\end{center}
\end{table}
\end{thm}
\begin{proof}
If $\End_A(W) \cong \R$, then $\langle \chi_W , \chi_W \rangle = 1$, so $W \res{}{}$ is simple, thus $W \res{}{} = V$ and $V$ is Realisable. By \ref{pr:cindres}, we have that $W \oplus W \cong W \res{}{}\ind{}{} \cong V \ind{}{}$. Furthermore, $W \res{}{} \oplus W \res{}{} \cong V \ind{}{} \res{}{} \cong V \oplus \vw \cdot \overline{V}$ by \ref{pr:cresind}, hence $V \cong \vw \cdot \overline{V}$.

If $\End_A(W) \cong \C$, then $\langle \chi_W , \chi_W \rangle = 2$. So $W\res{}{} \cong V \oplus U$, where $U$ is simple, and $W \oplus W \cong W \res{}{} \ind{}{} \cong V \ind{}{} \oplus U\ind{}{}$. $W$ is simple, so $W \cong V \ind{}{}$. Therefore, $W \res{}{} \cong V \ind{}{} \res{}{} \cong V \oplus \vw \cdot \overline{V}$, and $W \res{}{}$ is the sum of two simple modules, thus $V \not\cong \vw \cdot \overline{V}$.

If $\End_A(W) \cong \mathbb{H}$, then $\langle \chi_W , \chi_W \rangle = 4$. If $W \res{}{}$ was the sum of four non-isomorphic simple modules, then $\mathbb{H} \cong \End_A(W) \leq \End_{\C G }W \res{}{} \cong \C^4$, a contradiction:
the non-commutative algebra $\mathbb{H}$ cannot be a subalgebra of the commutative algebra  $\C^4$.
Thus $\End_{\C G }W \res{}{} \cong M_2(\C)$, so $W \res{}{} \cong V \oplus V$. Therefore, $V \oplus V \cong W \res{}{} \cong \res{}{}V \ind{}{} \cong V \oplus \vw \cdot \overline{V}$, so $V \cong \vw \cdot \overline{V}$, and $W \oplus W \cong W \res{}{}\ind{}{} \cong V \ind{}{} \oplus V \ind{}{}$, thus $W \cong V \ind{}{}$.

If the last two columns were Realisable, $V = U\res{}{}$ for some $U$. Then $W \cong V \ind{}{} \cong U \res{}{} \ind{}{}= U \oplus U$, which cannot happen as $W$ is simple. The final two rows follow from \ref{biltypes} and \ref{FS} respectively.
\end{proof}

The next corollary sheds light on the origin of irreducible A-representations.
\begin{cor}\label{irr_CGG_mod}
  Every simple $\CGG$-module $W$ occurs in $V \ind{}{}$ for some simple $\C G $-module $V$.
  This $V$ is unique if and only if $\End_{\R}W \neq \C$.
If $\End_{\R}W = \C$, then $W$ occurs twice as $V \ind{}{}$ and $(\vw \cdot \overline{V})\ind{}{}$.
\end{cor}

Corollary~\ref{irr_CGG_mod} allows us to derive the A-character table from the complex character table.
Go through the irreducible characters $\chi$ of $G$.
If $\F(\chi) = 1$, $\chi$ is an irreducible A-character.
If $\F(\chi) = -1$, $2\chi$ is an irreducible A-character.
If $\F(\chi) = 0$, then $\chi + \vw \cdot \overline{\chi}$ is an irreducible A-character.
This yields all irreducible A-characters, once for $\F(\chi) \neq 0$, and twice for $\F(\chi) = 0$.

The table gives a second proof of Theorem~\ref{prop:centredim}, which we find instructive:
\begin{proof}[Second proof of \ref{prop:centredim}]
Let $r,s,t$ be the number of irreducible A-representations of type $\R, \C, \mathbb{H}$ respectively, so $\dim_{\R}Z(\CGG) = r + 2s + t$. Each irreducible A-representation with endomorphism ring $\mathbb{H}$ or $\R$ comes from one irreducible complex representation, and those with endomorphism ring $\C$ come from two. So the number of complex irreducible representations is also equal to $r+2s+t$.
\end{proof}

\subsection{The A-character Table is Square}  This has been proved by Noohi and Young \cite[Cor. 3.12]{LTJ}:
\begin{thm}\label{square}
$\#($Irreducible A-Representations$) = \#($Real Conjugacy Classes$)$.
\end{thm}

In order to prove this, we make use of a theorem of Brauer. Let $S$ be the set of irreducible complex  characters of $G$, and $\CC$ the set of conjugacy classes of $G$.
\begin{prop}\label{braueraction}\cite[Thm. 11.9]{CR1}
Let $A$ be a group which acts on both $\CC$ and $S$ in such a way that $\alpha(\chi)_{\alpha(C)} = \chi_C$ for all $\chi \in S$, $\alpha \in A$, $C \in \CC$, where $\chi_C = \chi(\vg)$ for $\vg \in \CC$. Then each element of $A$ fixes the same number of irreducible characters of $G$ as conjugacy classes.
\end{prop}

\begin{lem}\label{doubleactionlem}
$\#($Conjugacy Classes of $G$ with $(\!(\vg)\!) = (\vg)_{G}) = \#($Irreducible A-representations of type $\R$ or $\mathbb{H})$.
\end{lem}
\begin{proof}
Let $\G$ act on $S$ by $\vz(\chi) \coloneqq \vz \cdot {}^{\pi(\vz)}\chi$. Now an odd $\vw$ fixes $\chi$ if and only if $\chi = \vw^{-1} \cdot \overline{\chi}$ which happens if and only if $\chi$ is not of complex type, as given in \ref{table}. 

Let $\G$ act on $\CC$ by $\vz((\vg)_G) = (\vz \vg^{\pi(\vz)} \vz^{-1})_{G}$. Then $\chi$ is fixed by an odd $\vw$ if and only if there is some $\vh \in G$ with $\vh^{-1}\vg \vh = \vw \vg^{-1} \vw^{-1}$, and this happens if and only if some odd element $\vz = \vh \vw$ has $\vg \vz = \vz \vg^{-1}$. Then by the criterion for Real conjugacy class splitting given in cases B1, B2 after \ref{lem:doubleaction}, this is the case if and only if $(\!(\vg)\!) = (\vg)_G$. Then for $\vg \in C \in \CC$, we have $\vz(\chi)_{\vz(C)} = \vz(\chi)(\vz \vg^{\pi(\vz)} \vz^{-1}) =\chi(\vz^2 \vg^{\pi(\vz)^2} \vz^{-2}) = \chi(\vg) = \chi_C$. Therefore, by \ref{braueraction} above, we have the result.
\end{proof}

\begin{cor}\label{diffcor}
$\#($Irreducible A-Representations of type $\C) = \# ($Conjugacy Classes$) - \# ($Real Conjugacy Classes$)$.
\end{cor}
\begin{proof}
  The difference on the right
  is the number of pairs of conjugacy classes of $G$ which join together to form a Real conjugacy class.
  By Lemma~\ref{doubleactionlem}, this is twice the number of irreducible characters of $G$ of complex type.
This is the left-hand side by Corollary~\ref{irr_CGG_mod}. 
\end{proof}

\begin{proof}[Proof of \ref{square}]
By \ref{prop:centredim}, \# (Conjugacy Classes) $= r + 2s + t$, so by \ref{diffcor}, \# (Real Conjugacy Classes) = \# (Conjugacy Classes) $-$ \#(Irreducible A-Representations of type $\C$) $= (r + 2s + t)  - s = r+s+t =$ \#(Irreducible A-Representations).
\end{proof}

We now list some immediate consequences.

\begin{cor}\label{formbasis}
Irreducible A-characters form a basis of the $\C$-vector space of Real class functions.
\end{cor}

\begin{cor}[Column Orthogonality]
  Let $\chi_1, \ldots ,\chi_n$ be irreducible A-characters,
  $\vg_1, \ldots, \vg_n \in G$ representatives of Real conjugacy classes.
If $\CR_{\G}(\vg)$ is the Real stabiliser (Section~\ref{Real_CC}), then
\[
\frac{\left| \CR_{\G}(\vg_s) \right| }{2} \delta_{rs} = \sum_{j=1}^n \frac{1}{m_j} \chi_j (\vg_r) \overline{\chi_j(\vg_s)}.
\]
\end{cor}


\subsection{More Results on Irreducible A-Representations}
Let us state more consequences of Theorem~\ref{table}.

\begin{cor}\label{1dimendo}
The endomorphism ring of an irreducible A-representation is $\R$ if and only if the underlying $\C G $-module is simple. In particular, all 1-dimensional A-representations have endomorphism ring $\R$.
\end{cor}

\begin{cor}
If $W$ is an irreducible A-representation with $\End_A(W) \not\cong \R$, then $\dim_{\C}W$ is even.
\end{cor}

\begin{prop}
The central primitive idempotent of $\CGG$ corresponding to the irreducible A-character $\chi_j$ is
\[
e_j = \frac{\chi_j (\ve)}{m_j \left| G \right|} \sum_{\vg \in G} \chi_j(\vg^{-1})\vg.
\]
\end{prop}

\begin{proof}
We know from \ref{lem:centralcrit} that $e_j$ is a $\C$-linear combination of $\vg \in G$. Hence, $e_j = \sum_{\vg \in G} a_{\vg} \vg$. The regular representation $\CGG$ has character
\[
\chi_R(\vg) = \twopartdef{2 \left| G \right|}{\vg = \ve,}{0}{\vg \neq \ve,}
\]
so that 
$\chi_R(e_j \vg^{-1}) = \chi_R(a_{\vg} \ve) = 2 \left| G \right| a_{\vg}$
for all $\vg \in G$.
If $W_j$ is the A-representation that affords the A-character $\chi_j$, then
\begin{align*}
\chi_R(e_j \vg^{-1}) = \sum_k \frac{2 \dim_{\C}W_k}{m_k}\chi_k(e_j \vg^{-1}) =\frac{2 \dim_{\C}W_j}{m_j}\chi_j(e_j \vg^{-1}) =\frac{2 \chi_j(\ve)}{m_j} \chi_j(\vg^{-1}),
\end{align*}
since $e_j$ acts as the identity on $W_j$ and as $0$ on $W_k$ for $j \neq k$.
\end{proof}

If $\chi_j$ is of type $\R$ or $\mathbb{H}$, then $e_j$ is a central primitive idempotent of $\C G$. 
If $\chi_j$ of complex type, $e_j$ is the sum of two primitive central idempotents of $\bC G$.

\begin{cor}
If, as $\R$-algebras, 
\[
\CGG \cong \prod_{i=1}^r M_{a_i}(\R) \times \prod_{j=1}^s M_{b_j}(\C) \times \prod_{k=1}^t M_{c_k}(\mathbb{H}),
\]
then, as $\bC$-algebras, 
\[
\C G  \cong \prod_{i=1}^r M_{\frac{a_i}{2}}(\C) \times \prod_{j=1}^s \left( M_{\frac{b_j}{2}}(\C) \times M_{\frac{b_j}{2}}(\C) \right) \times \prod_{k=1}^t M_{c_k}(\C).
\]
\end{cor}

\subsection{Applications of the Frobenius-Schur Indicator}
The results of this section all are consequences of Theorem~\ref{FS}.
Let us introduce another version of the Frobenius-Schur indicator, applicable to A-characters.
Let $W$ be an A-representation with A-character $\chi=\chi_W$.
Define $\F_A(W) = \F_A(\chi) \coloneqq \F(\chi_{W \res{\bC G}{\CGG}}) = \F(\chi)$.

\begin{prop}\label{Afstypes}
  If $W$ is an irreducible A-representation, then 
\begin{align*}
\F_A(W) =
\left\{
		\begin{array}{cl}
			1 & \mbox{if } W \mbox{ is of type } \bR, \\
			0 & \mbox{if } W \mbox{ is of type } \bC,  \\
			-2 & \mbox{if } W \mbox{ is of type } \bH.
		\end{array}
	\right.
\end{align*}
\end{prop}
\begin{proof} This follows immediately from Theorems~\ref{FS} and \ref{table}.
\end{proof}

This leads to a character theoretic criterion of irreducibility.
\begin{prop}\label{simpleAtest}
An A-representation with A-character $\chi$ is simple if and only if $\langle \chi, \chi \rangle + \F_A(\chi) = 2$.
\end{prop}
\begin{proof}
Decompose $W$ into irreducible A-representations: $W \cong \oplus_j S_j^{a_j}$. Then
\begin{align*}
\langle \chi, \chi \rangle + \F_A(W) = \sum_j a_j^2 m_j + \sum_j \F_A(S_j)
= \sum_j (a_j^2 m_j + \F_A(S_j)).
\end{align*}
Consider the term $a_j^2 m_j + \F_A(S_j)$: 
\begin{itemize}
\item
If $\End_A(S_j) \cong \R$, $a_j^2 m_j + \F_A(S_j) = a_j^2 + 1 \geq 2$.
\item
If $\End_A(S_j) \cong \C$, $a_j^2 m_j + \F_A(S_j) = 2a_j^2 \geq 2$.
\item
If $\End_A(S_j) \cong \mathbb{H}$, $a_j^2 m_j + \F_A(S_j) = 4a_j^2 -2\geq 2$.
\end{itemize}
Therefore, all these terms are positive and at least $2$. The sum is 2 if and only if it contains the single term $a_1^2 m_1 + \F_A(S_1)$ with $a_1 = 1$.
\end{proof}

Let $\chi_1,\ldots,\chi_n$ be all distinct irreducible complex characters of $G$.
\begin{prop}\label{sqsum1}
  Define $\fn{r}{G}{\N}$ by $r(\vh) = \left| \{\vz \in \G \setminus G \mid \vz^2 = \vh\} \right|$.
  Then
  $$
  r(\vh) = \sum_{j=1}^{n} \F(\chi_j) \chi_j(\vh)\, .
  $$
\end{prop}
\begin{proof}
Since $\fn{r}{G}{\N}$ is a class function, $r = \sum_j a_j \chi_j$ for some $a_j \in \C$. Then 
\begin{align*}
a_j &= \langle \chi_j , r \rangle
= \frac{1}{\left| G \right|} \sum_{\vh \in G} \# \{\vz \in \G \setminus G \mid \vz^2 = \vh\} \chi_j(\vh) \\
&= \frac{1}{\left| G \right|} \sum_{\vh \in G} \sum_{\substack{\vz \in \G \setminus G \\ \vz^2 = \vh}} \chi_j(\vz^2)
= \frac{1}{\left| G \right|} \sum_{\vz \in \G \setminus G} \chi_j(\vz^2) = \F(\chi_j). \qedhere
\end{align*}
\end{proof}

\begin{cor}\label{sqsum2}
$r(\ve) = \left| \{\vz \in \G \setminus G \mid \vz^2 = \ve\} \right| = \sum_{j=1}^n\F(\chi_j) \dim_{\C} (\chi_j)$.
\end{cor}

\begin{cor}\label{sqsum3}
  If $G$ has no A-representations of type $\bH$, then $\fn{r}{G}{\N}$ attains its maximum value at the identity.
\end{cor}
\begin{proof}
Let $\vg \in G$. Using the facts that $r(\vg) \in \R_{\geq 0}$ and $\F(\chi_j)\geq 0$, 
\[
r(\vg) = \left| \sum_{j \: : \: \F(\chi_j) = 1}  \chi_j(\vg) \right|
\leq \sum_{j \: : \: \F(\chi_j) = 1} \left| \chi_j(\vg) \right| \leq \sum_{j \: : \: \F(\chi_j) = 1} \left| \chi_j(\ve) \right| = r(\ve). \qedhere
\]
\end{proof}

\subsection{The Classical and the Split Cases} \label{classical_split}
For the standard Real structure, we can choose $\vw = (\ve,-1)$ as a fixed odd element.
Since $\vw$ is central, $V \cong \vw \cdot \overline{V}$ if and only if $V \cong \overline{V}$.
Furthermore, $\vw$-invariant bilinear form is simply a $G$-invariant bilinear form (see~\eqref{invariant_form}). 
Moreover, $\vw$-alternating/$\vw$-symmetric (cf. Theorem~\ref{biltypes}) are simply alternating/symmetric respectively because $\vw^2 = (\ve,1)$.
The results of this section fully generalise the classical theory
(cf. \cite[Ch. 9]{CR2} or \cite[Appendix]{RUM}).

For a split Real structure, we can choose a fixed odd element $\vw$ with $\vw^2=\ve$.
In particular, $\vw$-alternating/$\vw$-symmetric are still simply alternating/symmetric.
However, $\vw$ is no longer central, so the forms are no longer $G$-invariant.
This briefly illustrates the difficulties of going from the classical case to the split case, and then to the general case.

\section{The Real Structure $A_n \leq S_n$}\label{A<S}
Throughout this section, let $G \leq \G$ be  a $C_2$-graded group with all simple $\R \G$ modules of type $\R$.
This is called a \emph{totally orthogonal} Real structure. 
This ensures that all conjugacy classes of $\G$ are self-inverse
and we can use
Proposition~\ref{pr:selfinverseReal}:
\begin{cor}\label{Realcycle}
  If $\vg \in G$, then
  $(\!(\vg)\!) = \twopartdef{(\vg)_{G}}{(\vg)_{G} \text{ is not self-inverse,}}{(\vg)_{\G}}{(\vg)_{G} \text{ is self-inverse.}}$
\end{cor}

The symmetric Real structure on $A_n$ is totally orthogonal. Moreover, 
we can say precisely which conjugacy classes in $A_n$ are self-inverse.
\begin{lem}\label{lem:realalternatingcycle}
Let $\vg\in A_n$ have disjoint cycle decomposition with  cycle lengths $r_1, ... , r_k$. The following statements are equivalent:  
\begin{enumerate}[label=(\roman*)]
\item the class $(\vg)_{A_n}$ is not self-inverse,
\item   the  $r_j$ are distinct, each odd and $\sum_{j=1}^k \frac{r_j - 1}{2}$ is odd,
\item   the  $r_j$ are distinct, each odd, and the number of the $r_j$, congruent to $3 \! \mod 4$ is odd.
\end{enumerate}  
\end{lem}
\begin{proof}
Let $(\vg)_{A_n}$ be a self-inverse class. If $\vh \in A_n$ with $\vh \vg \vh^{-1} = \vg^{-1}$, then each of the distinct cycles of odd length $r_i$ of $\vg$ must be conjugate to its inverse by $\vh$, hence $\vh$ is a product of $\sum_{i=1}^k \frac{r_i - 1}{2}$ transpositions, so this quantity must be even. Conversely for $\vg$ with such a cycle type, define $\vh$ as above with $\vh \vg \vh^{-1} = \vg^{-1}$.
\end{proof}

Since all $\R \G$-modules are of type $\R$,
the only possible A-block structures (cf. \ref{possibleblocks}) are types I, III or V. The next lemma follows.
\begin{lem}
If $G\leq\G$ is totally orthogonal, then neither $\CGG$, nor $\bR G$ has any simple module of type $\bH$.
\end{lem}
For $A_n \leq S_n$, not only this recovers the classical result that $\R A_n$ has no simple modules of quaternionic type but also shows that $\C {\ast} S_n$ has no simple modules of quaternionic type.

Given an irreducible complex representation $V$ of $G$, 
consider the four representations $V$, $\overline{V}$, $\vw \cdot V$ and $\vw \cdot \overline{V}$.
In type I,  $(\F_{\bC}(V),\F(V))=(1,1)$ and all four are isomorphic.
In type III, $(\F_{\bC}(V),\F(V))=(1,0)$ and $V\cong\overline{V}\not\cong \vw \cdot V\cong\vw\cdot\overline{V}$.
In type V, $(\F_{\bC}(V),\F(V))=(0,1)$ and  $V\cong\vw \cdot \overline{V}\not\cong \overline{V}\cong\vw \cdot {V}$.

It is well-known that $\R A_n$ does not have a simple module of type $\bC$ if and only if $n \in \{2,5,6,10,14\}$ \cite[Cor. 4.9]{ALT}. We can understand this for $\C {\ast} S_n$ now.
\begin{prop}\label{compalttype}
$A_n \leq S_n$ has no irreducible A-representation of complex type if and only if $n \in \{2,3,4,7,8,12\}$.
\end{prop}
\begin{proof}
By  Corollary~\ref{diffcor}, existence an irreducible A-representation of type $\bC$ is equivalent to existence of a partition of $n$ of distinct odd lengths $r_j$, where the number of $r_j \equiv 3 \! \mod 4$ is even. For $n\leq 12$ one can list the partitions, and such a partition can be found for all $n \geq 13$ by considering the cases of $n \!\! \mod 4$. 
\end{proof}

We finish the paper by joining together the analysis of the possible A-blocks in the totally orthogonal case with Proposition~\ref{sqsum1}.
Pick an element $\vg \in G$.
Using the notation of Proposition~\ref{sqsum1}, the number of elements of $\G\setminus G$ squaring to $\vg$ minus the number of elements of $G$ squaring to $\vg$ is given by:
\begin{align}
  \sum_{j=1}^n \F(\chi_j) \chi_j(\vg) - \sum_{j=1}^n \F_{\C}(\chi_j) \chi_j(\vg) =
  \sum_{j\; :\; \chi_j = \overline{\chi_j}} \chi(\vg) - \sum_{j\; :\;\vw \cdot \chi_j = \overline{\chi_j}} \chi_j(\vg).
\end{align}

\appendix
\section{Post-print corrections}
We would like to thank Hugh Osborn for spotting some typos in the final version. Hereby we correct them.

In comment (X) to Equation~\eqref{sqb_AC} the group $Q_8 \leq Q_8\rtimes G_{32}^8$ is a typo. It should be
$Q_8\times C_{2}\leq G_{32}^8$ as stated in Table~\ref{table_block_structure}.

In comment (IX) to Equation~\eqref{uqb_AC} there is a confusion between two different groups.
The semidirect product  $Q_8\rtimes C_2$ is the group $G_{16}^{8}$.
It is different from the Pauli group $G_{16}^{13}$.
Both groups contain $Q_8$ as a subgroup of index 2.
For both groups the faithful representation $\bC^2$ of $Q_8$ yields an A-block of type (IX) with $n=1$.
These two are 
the only two groups of order 16 that have an A-representation of type (IX).

\bibliography{arepsbiblio}{}
\bibliographystyle{plain}

\end{document}